\documentclass[12pt]{amsart}

\usepackage{amsmath,amsfonts,amsthm}

\renewcommand{\theequation}{\thesection.\arabic{equation}}
\newtheorem{theorem}{Theorem}
\newtheorem{lemma}{Lemma}
\newtheorem{proposition}{Proposition}
\newtheorem{corollary}{Corollary}
\newtheorem*{remark*}{Remark}
\newtheorem*{remarks}{Remarks}
\newtheorem{definition}{Definition}

\newtheorem*{assumptions}{Assumptions}

\newcommand{\eqnsection}{
\renewcommand{\theequation}{\thesection.\arabic{equation}}
    \makeatletter
    \csname  @addtoreset\endcsname{equation}{section}
    \makeatother}
\eqnsection

\def\z{{\mathbb Z}}

\def\ee{\mathrm{e}}
\def\d{\, \mathrm{d}}

\def\eps{\varepsilon}
\def\R{{\mathbb R}}
\def\Z{{\mathbb Z}}

\def\C{{\mathbb C}}
\def\N{{\mathbb N}}

\author[N. Enriquez]{Nathana\"el ENRIQUEZ}
\address{Laboratoire Modal'X, Universit\'e Paris 10, 200
Avenue de la R\'epublique, 92000 Nanterre, France}
\address{Laboratoire de Probabilit\'es et Mod\`eles Al\'eatoires, CNRS
UMR 7599,
Universit\'e Paris 6, 4
place Jussieu, 75252 Paris Cedex 05, France}
\email{nenriquez@u-paris10.fr}

\author[C. Sabot]{Christophe SABOT}
\address{Institut Camille Jordan, CNRS UMR 5208, Universit\'e de
Lyon, Universit\'e Lyon 1, 43, Boulevard du 11 novembre 1918,
69622 Villeurbanne Cedex} \email{sabot@math.univ-lyon1.fr}

\author[L. Tournier]{Laurent TOURNIER}
\address{Institut Camille Jordan, CNRS UMR 5208, Universit\'e de
Lyon, Universit\'e Lyon 1, 43, Boulevard du 11 novembre 1918,
69622 Villeurbanne Cedex} \email{tournier@math.univ-lyon1.fr}

\author[O. Zindy]{Olivier ZINDY}
\address{Laboratoire de Probabilit\'es et Mod\`eles Al\'eatoires, CNRS UMR 7599, Universit\'e Paris 6, 4 place Jussieu, 75252 Paris Cedex 05, France}
\email{olivier.zindy@upmc.fr}

\keywords{Random walk in random environment, stable laws,
fluctuation theory of random walks, Beta distributions}
\subjclass[2010]{primary 60K37, 60F05, 82B41; secondary 60E07, 60E10}

\thanks{This research was supported by the french ANR project MEMEMO}

\title[Stable fluctuations for ballistic RWRE]{Stable fluctuations for ballistic random walks in random environment on $\z$}

\newcommand{\indic}[1]{{\bf 1}_{\{#1\}}}
\def\ind{{\bf 1}}

\newcommand{\tia}{\tau_{\rm IA}}

\usepackage{comment}

\def\taut{{\widetilde{\tau}}}

\def\Var{\textit{Var}}

\def\bE{\mathbb{E}}
\def\bP{\mathbb{P}}
\def\bVar{\mathbb{V}{\rm ar}}

\def\bPp{\bP^{\geq 0}}
\def\Pp{P^{\geq 0}}
\def\bEp{\bE^{\geq 0}}
\def\Ep{E^{\geq 0}}
\def\bVarp{\bVar^{\geq 0}}
\def\Varp{\Var^{\geq 0}}

\def\eqd{\begin{eqnarray*}}
\def\eqf{\end{eqnarray*}}

\def\hseq{\hspace{-.5cm}} 

\def\nonoverlap{NO(n)}
\def\overlap{NO(n)^c}
\def\nonoverlapc{\widehat{NO}(n)}

\def\beq{\begin{equation}} 
\def\eeq{\end{equation}} 

\newcommand{\tautx}[1]{\taut^{(#1)}}

\newcommand{\limites}[2]{\overset{#1}{\underset{#2}{\longrightarrow}}}

\newcommand{\defeq}{\mathrel{\mathop:}=}
\newcommand{\defqe}{=\mathrel{\mathop:}}

\begin{document}

\maketitle

\bigskip

{\footnotesize \noindent{\slshape\bfseries Abstract.} We consider transient random walks in random environment on $\Z$ in the positive speed (ballistic) and critical zero speed regimes. A classical result of Kesten, Kozlov and Spitzer proves that the hitting time of level $n$, after proper centering and normalization, converges to a completely asymmetric stable distribution, but does not describe its scale parameter. Following~\cite{limitlaws}, where the (non-critical) zero speed case was dealt with, we give a new proof of this result in the subdiffusive case that provides a complete description of the limit law. Like in \cite{limitlaws}, the case of Dirichlet environment turns out to be remarkably explicit. }

\bigskip

\section{Introduction}
   \label{s:intro}

Random walks in a one-dimensional random environment were first introduced in the late sixties as a toy model for DNA replication. The recent development of micromanipulation technics such as DNA unzipping has raised a renewed interest in this model in genetics and biophysics, cf.\ for instance \cite{cocco} where it is involved in a DNA sequencing procedure. Its mathematical study was initiated by Solomon's 1975 article \cite{solomon}, characterizing the transient and recurrent regimes and proving a strong law of large numbers. A salient feature emerging from this work was the existence of an intermediary regime where the walk is transient with a zero asymptotic speed, in contrast with the case of simple random walks. Shortly afterward, Kesten, Kozlov and Spitzer \cite{kks} precised this result in giving limit laws in the transient regime. When suitably normalized, the (properly centered) hitting time of site $n$ by the random walk was proved to converge toward a stable law as $n$ tends to infinity, which implies a limit law for the random walk itself. In particular, this entailed that the ballistic case (i.e.\ with positive speed) further decomposes into a diffusive and a subdiffusive regimes. 

The aim of this article is to fully characterize the limit law in the subdiffusive (non-Gaussian) regime. Our approach is based on the one used in the similar study of the zero speed regime \cite{limitlaws} by three of the authors. The proof of \cite{kks} relied on the use of an embedded branching process in random environment (with immigration), which gives little insight into the localization of the random walk and no explicit parameters for the limit. Rather following Sinai's study \cite{sinai} of the recurrent case and physicists' heuristics developed since then (cf.\ for instance \cite{bouchaud}), we proceed to an analysis of the potential associated to the environment as a way to locate the ``deep valleys'' that are likely to slow down the walk the most. We thus prove that the fluctuations of the hitting time of $n$ with respect to its expectation mainly come from the time spent at crossing a very small number of deep potential wells. Since these are well apart, this translates the situation to the study of an almost-i.i.d.\ sequence of exit times out of ``deep valleys''. The distribution of these exit times involves the expectation of some functional of a meander associated to the potential, which was shown in \cite{renewal} to relate to Kesten's renewal series, making it possible to get explicit constants in the limit. The case of Beta distributions turns out to be fully explicit as a consequence of a result by Chamayou and Letac \cite{chamayou-letac}. The proof also covers the non-ballistic regime, including the critical zero-speed case, which was not covered in \cite{limitlaws}. 

Let us mention two other works relative to this setting. Mayer-Wolf, Roitershtein and Zeitouni~\cite{mayerwolf-roitershtein-zeitouni} generalized the limit laws of \cite{kks} from i.i.d.~to Markovian environment, still keeping with the branching process viewpoint. And Peterson \cite{peterson} (following \cite{peterson-zeitouni}), in the classical i.i.d.~setting and using potential technics, proved that no quenched limit law (i.e.~inside a fixed generic environment) exists in the ballistic subdiffusive regime. 

The paper is organized as follows. Section \ref{s:hyp+thm} states the results. The notions of excursions and deep valleys are introduced in Section \ref{sec:valleys}, which will enable us to give in Subsection \ref{subsec:sketch} the sketch and organization of the proof that occupies the rest of the paper. 

\section{Notations and main results}
   \label{s:hyp+thm}

Let $\omega\defeq (\omega_i, \, i \in \Z)$ be a family of i.i.d. random
variables taking values in $(0,1)$ defined on $\Omega,$ which stands
for the random environment. Denote by $P$ the distribution of
$\omega$ and by $E$ the corresponding expectation. Conditioning on
$\omega$ (i.e.\ choosing an environment), we define the random walk
in random environment $X\defeq (X_n, \, n \ge 0)$ starting from $x\in\Z$ as a nearest-neighbour
random walk on $\Z$ with transition probabilities given by $\omega$: if we denote by $P_{x,\omega}$ the law of the Markov chain $(X_n, \, n \ge 0)$ defined by $P_{x,\omega} \left( X_{0} = x\right) =1$ and
 \begin{align*}
P_{x,\omega} \left( X_{n+1} = z \, | \, X_n =y\right) \defeq \left\{ 
\begin{array}{lll}
		\omega_y, & {\rm if}\ z=y+1,\\
		1-\omega_y, & {\rm if}\ z=y-1,\\
		0, & {\rm otherwise,}
\end{array}
\right.
 \end{align*}
then the joint law of $(\omega,X)$ is $\bP_x(\d \omega,\d X)\defeq P_{x,\omega}(\d X)P(\d \omega)$.
For convenience, we let $\bP\defeq \bP_0$. We refer to \cite{zeitouni} for an overview of results on random walks in random environment. 
An important role is played by the sequence of variables 
 \begin{align}
 \label{defrho}
\rho_i\defeq  \frac{1-\omega_i}{\omega_i}, \qquad i \in \Z.
 \end{align}
We will make the following assumptions in the rest of this paper. 
\begin{assumptions}\leavevmode
\begin{itemize}
  \item[({\it a})] there exists $0<\kappa<2$ for which $E \left[  \rho_0^{\kappa} \right]=1$ and $E \left[  \rho_0^{\kappa} \log^+ \rho_0 \right]<\infty$;
  \item [({\it b})] the  distribution  of $\log \rho_0$  is non-lattice.
\end{itemize}
\end{assumptions}

We now introduce the hitting time $\tau(x)$ of site $x$ for the
random walk $(X_n, \, n \ge 0),$
\begin{equation*}\label{hittimerw}
    \tau(x)\defeq  \inf \{ n \ge 1: \, X_n=x \}, \qquad x \in \Z.
\end{equation*}

For $\alpha \in (1,2)$, let $\mathcal{S}_{\alpha}^{ca}$ be
the completely asymmetric stable zero mean random variable of index $\alpha$
with characteristic function 
\begin{equation}\label{eq:defStable}
E[\ee^{it  \mathcal{S}_{\alpha}^{ca}}]=\exp((-it)^\alpha)= \exp\left(|t|^\alpha\cos\frac{\pi\alpha}{2}\left(1-i\operatorname{sgn}(t)\tan\frac{\pi\alpha}{2}\right)\right),
\end{equation}
where we use the principal value of the logarithm to define $(-it)^\alpha(=\ee^{\alpha\log(-it)})$ for real $t$, and $\operatorname{sgn}(t)\defeq\ind_{(0,+\infty)}(t)-\ind_{(-\infty,0)}(t)$. Note that $\cos\frac{\pi\alpha}{2}<0$. 

For $\alpha=1$, let $\mathcal{S}_1^{ca}$ be the completely asymmetric stable random variable of index $1$ with characteristic function
\begin{equation}\label{eq:defStable1}
E[\ee^{i t\mathcal{S}_1^{ca}}]=\exp(-\frac{\pi}{2}|t|-it\log|t|)=\exp\big(-\frac{\pi}{2}|t|(1+i\frac{2}{\pi}\operatorname{sgn}(t)\log|t|)\big). 
\end{equation}

Moreover, let us introduce the constant $C_K$ describing the tail of Kesten's renewal series $R\defeq \sum_{k\geq0}\rho_0\cdots\rho_k$, see \cite{kesten73}:
\begin{align*}
 P(R>x)\sim C_K  x^{-\kappa}, \qquad x \to
\infty.
\end{align*}
Note that several probabilistic representations are available to compute $C_K$ numerically, which are equally efficient. The first one was obtained by Goldie \cite{goldie}, a second  was conjectured by Siegmund \cite{siegmund}, and a third one was obtained in \cite{renewal}. 

The main result of the paper can be stated as follows. The symbol ``$\stackrel{\mathrm{(law)}}{\longrightarrow}$'' denotes the convergence in distribution.
\medskip
 \begin{theorem}
\label{T:MAIN} Under assumptions ({\it a}) and ({\it b}) we have, under $\bP$, when $n$ goes to infinity,
\begin{itemize}
	\item if $1<\kappa<2$, letting $v\defeq\frac{1-E[\rho_0]}{1+E[\rho_0]}$, 
 \begin{align}
  \frac{\tau(n)-nv^{-1}}{n^{1/\kappa}}
  	&\stackrel{\mathrm{(law)}}{\longrightarrow} 2\left(-\frac{\pi\kappa^2}{\sin(\pi\kappa)}C_K^2E[\rho_0^\kappa\log\rho_0]\right)^{1/\kappa}\mathcal{S}_{\kappa}^{ca}  \label{eqn:thm_tau} \\
\intertext{and}
  \frac{X_n-nv}{n^{1/\kappa}}
  	&\stackrel{\mathrm{(law)}}{\longrightarrow}  - 2\left(-\frac{\pi\kappa^2}{\sin(\pi\kappa)}C_K^2E[\rho_0^\kappa\log\rho_0]\right)^{1/\kappa} v^{1+\frac{1}{\kappa}}\mathcal{S}_{\kappa}^{ca};\label{eqn:thm_x}
\end{align}
	\item if $\kappa=1$, for some deterministic sequences $(u_n)_n,(v_n)_n$ converging to~1,
 \begin{align}
  \frac{\tau(n)-u_n\frac{2}{E[\rho_0\log\rho_0]} n\log n}{n}
  	&\stackrel{\mathrm{(law)}}{\longrightarrow} \frac{2}{E[\rho_0\log\rho_0]}\mathcal{S}_{1}^{ca}  \label{eqn:thm_tau1} \\
\intertext{and}
  \frac{X_n-v_n\frac{E[\rho_0\log\rho_0]}{2}\frac{n}{\log n}}{n/(\log n)^2}
    &\stackrel{\mathrm{(law)}}{\longrightarrow}  \frac{E[\rho_0\log\rho_0]}{2}\mathcal{S}_{1}^{ca}.\label{eqn:thm_x1}
\end{align}
\end{itemize}
In particular, for $\kappa=1$, the following limits in probability hold:
\begin{align}
\frac{\tau(n)}{n\log n}
	 \limites{(p)}{}\frac{2}{E[\rho_0\log\rho_0]}
\qquad\text{and}\qquad
\frac{X_n}{n/\log n}
	 \limites{(p)}{}\frac{E[\rho_0\log\rho_0]}{2}.
\end{align}
\end{theorem}

\begin{remarks}\leavevmode
\begin{itemize}
	\item The proof of the theorem will actually give an  expression for the sequence $(u_n)_n$. 
	\item The case $0<\kappa<1$, already settled in~\cite{limitlaws}, also follows from (a subset of) the proof. 
\end{itemize}
\end{remarks}
\medskip

This theorem takes a remarkably explicit form in the case of Dirichlet environment, i.e.\ when the law of $\omega_0$ is  $\mathrm{Beta}(\alpha,\beta)\defeq\frac{1}{B(\alpha,\beta)}x^{\alpha-1}(1-x)^{\beta-1}{\bf 1}_{(0,1)}(x) \d x,$ with $\alpha, \beta >0$ and $B(\alpha,\beta)\defeq \int_{0}^1 x^{\alpha-1}(1-x)^{\beta-1} \d x=\frac{\Gamma(\alpha)\Gamma(\beta)}{\Gamma(\alpha+\beta)}$. An easy computation leads to $\kappa=\alpha-\beta.$
Thanks to a very nice result  of Chamayou and Letac \cite{chamayou-letac} giving the explicit value of $C_K$ in this case, we obtain the following corollary.
\medskip

 \begin{corollary}
\label{c:main} In the case where $\omega_0$ has a distribution $\mathrm{Beta}(\alpha,\beta),$ with $1\leq \alpha-\beta <2,$ Theorem \ref{T:MAIN} applies with $\kappa=\alpha-\beta.$ Then we have, when $n$ goes to infinity, if $1<\alpha-\beta<2$,
 \begin{align*}
 \frac{\tau(n)-\frac{\alpha+\beta-1}{\alpha-\beta-1}n}{n^{\frac{1}{\alpha-\beta}}}
	 &\stackrel{\mathrm{law}}{\longrightarrow} 2\left(-\frac{\pi}{\sin(\pi(\alpha-\beta))}\frac{\Psi(\alpha)-\Psi(\beta)}{B(\alpha,\beta)^2}\right)^{\frac{1}{\alpha-\beta}} \mathcal{S}_{\alpha-\beta}^{ca},\\
\intertext{and}
  \frac{X_n-\frac{\alpha-\beta-1}{\alpha+\beta-1}n}{n^{\frac{1}{\alpha-\beta}}}
	 &\stackrel{\mathrm{law}}{\longrightarrow} -2\left(-\frac{\pi}{\sin(\pi(\alpha-\beta))}\frac{\Psi(\alpha)-\Psi(\beta)}{B(\alpha,\beta)^2}\right)^{\frac{1}{\alpha-\beta}} \left(\tfrac{\alpha-\beta-1}{\alpha+\beta-1}\right)^{1+\frac{1}{\alpha-\beta}}\mathcal{S}_{\alpha-\beta}^{ca},
\end{align*}
where $\Psi$ denotes the classical digamma function,
$\Psi(z)\defeq (\log \Gamma)'(z)=\frac{\Gamma'(z)}{\Gamma(z)}.$ Furthermore, if $\alpha-\beta=1$, then we have 
\begin{equation*}
E[\rho_0\log\rho_0]=\frac{B(\beta,\beta)}{2\beta}.
\end{equation*}
\end{corollary}
\bigskip

In the following, the constant $C$ stands for a positive constant large enough, whose value can change from line to line. We henceforth assume that hypotheses ({\it a}) and ({\it b}) hold; in particular, wherever no further restriction is mentioned, we have $0<\kappa<2$. 

\section{Notion of valley -- Proof sketch} \label{sec:valleys} 
Following Sinai \cite{sinai} (in the recurrent case), and more recently the study of the case $0<\kappa<1$ in \cite{limitlaws}, we define notions of potential and valleys that enable to visualize where the random walk spends most of its time. 

\subsection{The potential}
The potential, denoted by $V= (V(x), \; x\in \Z)$, is a function of the environment
$\omega$ defined by $V(0)=0$ and $\rho_x=\ee^{V(x)-V(x-1)}$ for every $x\in\Z$, i.e. 
\[
V(x) \defeq \left\{\begin{array}{lll}
	\sum_{i=1}^x \log \rho_i & {\rm if} \ x \ge 1,\\
	0 &  {\rm if} \ x=0,\\
	-\sum_{i=x+1}^0 \log \rho_i &{\rm if} \ x\le -1,
\end{array}
\right.
\]
where the $\rho_i$'s are defined in \eqref{defrho}. Under hypothesis ({\it a}), Jensen's inequality gives $E[\log \rho_0^\kappa]\leq \log E[\rho_0^\kappa]=0$, and hypothesis ({\it b}) excludes the equality case $\rho_0=1$ a.s., hence $E[\log \rho_0]< 0$ and thus $V(x)\to\mp\infty$ a.s.~when $x\to\pm\infty$. 

Furthermore, we consider the weak descending ladder epochs
of the potential, defined by $e_0\defeq 0$ and
\begin{equation} \label{eqn:def_e_i}
  e_{i+1} \defeq  \inf \{ k > e_i: \; V(k) \le V(e_i)\}, \qquad i \ge 0.
\end{equation}
Observe that $(e_i-e_{i-1})_{i \ge 1}$ is a family of i.i.d.~random variables.
Moreover, hypothesis ({\it a}) of Theorem \ref{T:MAIN} implies that $e_1$ is exponentially integrable. Indeed, for all $n>0$, for any $\lambda>0$, $P(e_1>n)\leq P(V(n)>0)=P(\ee^{\lambda V(n)}>1)\leq E[\ee^{\lambda V(n)}]= E[\rho_0^\lambda]^n$, and $E[\rho_0^\lambda]<1$ for any $0<\lambda<\kappa$ by convexity of $s\mapsto E[\rho_0^s]$. 

It will be convenient to extend the sequence $(e_i)_{i\geq 0}$ to negative indices by letting
\begin{equation} \label{eqn:def_e_negatifs}
e_{i-1}\defeq \sup\{k< e_i:\; \forall l<k, V(l)\geq V(k)\}, \qquad i\leq 0.
\end{equation}
The structure of the sequence $(e_i)_{i\in\Z}$ will be better understood after Lemma \ref{lem:e_negatifs}. 

Observe that the intervals $(e_i,e_{i+1}], {i \in \Z},$ stand for the excursions of the potential above its past minimum, provided $V(x)\geq 0$ when $x\leq 0$. 
Let us introduce $H_i,$ the height of the excursion $(e_i,e_{i+1}]$, defined by 
\begin{align*}
H_i \defeq  \max_{e_i \le k \le e_{i+1}} \left(V(k)-V(e_i)\right), \qquad i\in\Z.
\end{align*}
Note that the random variables $(H_i)_{i\ge 0}$ are i.i.d. For notational convenience, we will write $H\defeq H_0$. 

In order to quantify what ``high excursions'' are, we need a key result of Iglehart  \cite{igle} which gives the tail probability of $H$, namely 
\begin{equation}\label{iglehartthm}
P (H > h) \sim C_I \, \ee^{- \kappa h}, \qquad h \to \infty,
\end{equation}
where
\begin{align*}\label{cst:Iglehart}
C_I\defeq{\frac{(1-E[\ee^{\kappa V(e_1)}])^2}{\kappa E[\rho_0^\kappa\log\rho_0]E[e_1]}}.
\end{align*}
This result comes from the following classical consequence of renewal theory (see \cite{feller}): if $S\defeq\sup_{k\geq 0} V(k)$, then
\begin{equation}\label{eqn:feller}
P(S>h) \sim C_F\,\ee^{-\kappa h},\qquad h\to\infty,
\end{equation}
where $C_F$ satisfies $C_I=(1-E[\ee^{\kappa V(e_1)}])C_F$. 

\subsection{The deep valleys}
\label{deepvalleys}

The notion of deep valley is relative to the space scale. Let $n\geq 2$. To define the corresponding deep valleys, we extract from the excursions of the potential above its minimum these whose heights are greater than a critical height $h_n,$ defined by
\begin{equation*}
\label{hcritic} h_n\defeq \frac{1}{\kappa}\log n - \log\log n. 
\end{equation*}
Moreover, let $q_n$ denote the probability that the height of such an excursion is larger than $h_n$. Due to \eqref{iglehartthm}, it satisfies
\begin{align*}
\label{eq:defqn}
q_n \defeq  P(H > h_n ) \sim C_I \, \ee^{- \kappa h_n}, \qquad n \to \infty.
\end{align*}

Then, let $(\sigma(i))_{i \ge 1}$ be the sequence of the indices of the successive excursions whose heights are greater than $h_n.$ More precisely,
\begin{align*}
\sigma(1)&\defeq \inf \{ j \ge 0 : H_j \ge h_n\},\\
\sigma(i+1)&\defeq \inf \{ j >\sigma(i): H_j \ge h_n \}, \qquad i \ge 1.
\end{align*}
We consider now some random variables depending only on
the environment, which define the deep valleys.
\medskip
\begin{definition}
 \label{defvalley}
For $i \ge 1,$ let us introduce
\begin{equation*}
a_i\defeq e_{\sigma(i)-D_n},\qquad b_i\defeq e_{\sigma(i)},\qquad {d}_i\defeq e_{\sigma(i)+1},
\end{equation*}
where
\begin{equation}\label{eqn:def_d_n}
D_n\defeq \left\lceil\frac{1+\gamma}{A \kappa}\log n\right\rceil,
\end{equation}
with arbitrary $\gamma>0$, and $A$ equals $E[-V(e_1)]$ if this expectation is finite and is otherwise an arbitrary positive real number. For every $i\geq 1$, the piece of environment $(\omega_{x})_{a_i< x\leq d_i}$ is called the $i$-th deep valley (with bottom at $b_i$). 
\end{definition}
Note that the definitions of $a_i$ and $d_i$ differ slightly from those in \cite{limitlaws}. 
We shall denote by $K_n$ the number of such deep valleys before $e_n,$ i.e. \begin{equation*}
K_n\defeq  \# \{ 0\leq i\leq n-1 : H_i\geq h_n \}.
\end{equation*}

\begin{remark*}
In wider generality, our proof adapts easily if we choose $h_n, D_n$ such that $n\ee^{-\kappa h_n}\to +\infty$, $D_n\geq Ch_n$ for a large $C$, and $n\ee^{-2\kappa h_n}D_n\to 0$. These conditions ensure respectively that the first $n$ deep valleys include the most significant ones, that they are wide enough (to the left) so as to make negligible the time spent on their left after the walk has reached their bottom, and that they are disjoint. A typical range for $h_n$ is $\frac{1}{2\kappa}\log n+(1+\alpha)\log\log n\leq h_n \leq \frac{1}{\kappa}\log n-\eps\log\log n$, where $\alpha,\eps>0$. 
\end{remark*}


\subsection{Proof sketch}\label{subsec:sketch}

The idea directing our proof of Theorem~\ref{T:MAIN} is that the time $\tau(e_n)$ splits into ({\it a}) the time spent at crossing the numerous ``small'' excursions, which will give the first order $nv^{-1}$ (or $n\log n$ if $\kappa=1$) and whose fluctuations are negligible on a scale of $n^{1/\kappa}$, and ({\it b}) the time spent inside deep valleys, which is on the order of $n^{1/\kappa}$, as well as its fluctuations, and will therefore provide the limit law after normalization. Moreover, with overwhelming probability, the deep valleys are disjoint and the times spent at crossing them may therefore be treated as independent random variables. 

The proof divides into three parts: reducing the time spent in the deep valleys to an i.i.d.~setting (Section \ref{sec:iid_valleys}); neglecting the fluctuations of the time spent in the shallow valleys (Section \ref{sec:interarrival}); and evaluating the tail probability of the time spent in \emph{one} valley (Section \ref{sec:propplougon}). These elements shall indeed enable us to apply a classical theorem relative to i.i.d.~heavy-tailed random variables (Section \ref{sec:preuve_thm}). Before that, a few preliminaries are necessary.

\section{Preliminaries} \label{sec:preliminaries}

This section divides into three independent parts. The first part recalls usual formulas about random walks in a one-dimensional potential. The second one adapts the main results from \cite{renewal} in the present context. Finally the last part is devoted to the effect of conditioning the potential on $\Z_-$ (bearing in mind that this half of the environment has little influence on the random walk), which is a technical tool to provide stationarity for several sequences. 

In the following, for any event $A$ on the environments such that $P(A)>0$, we use the notations
\begin{equation*}
P^A\defeq P(\, \cdot \, \vert \, A)\quad\text{ and }\quad \bP^A\defeq\bP(\,\cdot\,|\,A)=P_\omega\times P^A(d\omega). 
\end{equation*}
In addition, the specific notations
\begin{equation*}
\Pp\defeq P(\, \cdot\, |\, \forall k\leq 0, V(k)\geq 0)\quad\text{ and }\quad\bPp\defeq P_\omega\times\Pp(d\omega)
\end{equation*}
will prove themselves convenient. 

\subsection{Quenched formulas}\label{subsec:quench}

We recall here a few Markov chain formulas that are of repeated use throughout the paper. 




\paragraph*{Quenched exit probabilities}
For any $a\leq x\leq b$, (see \cite{zeitouni}, formula (2.1.4))
\beq \label{eqn:zeitouni_p}
P_{x,\omega}(\tau(b)<\tau(a))=\frac{\sum_{a\leq k< x}\ee^{V(k)}}{\sum_{a\leq k<b}\ee^{V(k)}}.
\eeq
In particular, 
\beq \label{eqn:zeitouni_p2}
P_{x,\omega}(\tau(a)=\infty)=\frac{\sum_{a\leq k<x}\ee^{V(k)}}{\sum_{k\geq a}\ee^{V(k)}} 
\eeq
and
\beq \label{eqn:zeitouni_p3}
P_{a+1,\omega}(\tau(a)=\infty)=\bigg(\sum_{k\geq a}\ee^{V(k)-V(a)}\bigg)^{-1}. 
\eeq
Thus $P_{0,\omega}(\tau(1)=\infty)=\left(\sum_{k\le 0}\ee^{V(k)}\right)^{-1}\!=0$, $P$-a.s.~because $V(k)\to +\infty$ a.s.\ when $k\to-\infty$, and $P_{1,\omega}(\tau(0)=\infty)=\left(\sum_{k\ge 0}\ee^{V(k)}\right)^{-1}\!>0$, $P$-a.s.\ by the root test (using $E[\log\rho_0]<0$). This means that $X$ is transient to $+\infty$ $\bP$-a.s.

\paragraph*{Quenched expectation} For any $a<b$, $P$-a.s., (cf.~\cite{zeitouni})
\begin{align}
E_{a,\omega}[\tau(b)]
	& =  \sum_{a\leq j<b}\sum_{i\leq j}(1+\ee^{V(i)-V(i-1)})\ee^{V(j)-V(i)}\notag\\
	& =  \sum_{a\leq j<b}\sum_{i\leq j}\alpha_{ij}\ee^{V(j)-V(i)} \label{eqn:zeitouni_e}
\end{align}
where $\alpha_{ij}=2$ if $i<j$, and $\alpha_{jj}=1$. Thus, we have 
\beq \label{eqn:zeitouni_e2}
E_{a,\omega}[\tau(b)]
	\leq 2\sum_{a\leq j<b}\sum_{i\leq j}\ee^{V(j)-V(i)}
\eeq 
and in particular 
\beq \label{eqn:zeitouni_e3}
E_{a,\omega}[\tau(a+1)]=1+2\sum_{i< a}\ee^{V(a)-V(i)}\leq 2\sum_{i\leq a}\ee^{V(a)-V(i)}. 
\eeq
\paragraph*{Quenched variance} For any $a<b$, $P$-a.s., (cf.~\cite{alili} or \cite{goldsheid})
\begin{align}
\Var_{a,\omega}(\tau(b))
	& =  4\sum_{a\leq k<b}\sum_{j\leq k}\ee^{V(k)-V(j)}(1+\ee^{V(j-1)-V(j)})\bigg(\sum_{l<j}\ee^{V(j)-V(l)}\bigg)^2 \label{eqn:alili}\\
	& =  4\sum_{a\leq k<b}\sum_{j<k}(\ee^{V(k)+V(j)}+\ee^{V(k)+V(j+1)})\bigg(\sum_{l\leq j}\ee^{-V(l)}\bigg)^2.\notag
\end{align}
Thus, we have 
\begin{align}
\Var_{a,\omega}(\tau(b))
	& \leq  8\sum_{a\leq k< b}\sum_{j\leq k}\ee^{V(k)+V(j)}\bigg(\sum_{l\leq j}\ee^{-V(l)}\bigg)^2\notag\\
	& \leq  16\sum_{a\leq k< b}\sum_{l'\leq l\leq j\leq k}\ee^{V(k)+V(j)-V(l)-V(l')}. \label{eqn:alili2}
\end{align}

\subsection{Renewal estimates}

In this section we recall and adapt results from \cite{renewal}, which are very useful to bound finely the expectations of exponential functionals of the potential.


Let $R_-\defeq\sum_{k\leq 0}\ee^{-V(k)}$. Then Lemma 3.2 from~\cite{renewal}  proves that
\begin{equation} \label{eqn:r-}
\Ep[R_-]<\infty
\end{equation}
and that more generally all the moments of $R_-$ are finite under $\Pp$. 

Let us define
\begin{equation*}
T_H\defeq \min\{x\geq 0:\;V(x)= H\},
\end{equation*}
and
\begin{equation*}
M_1\defeq\sum_{k<T_H}\ee^{-V(k)},\qquad M_2\defeq\sum_{0\leq k<e_1}\ee^{V(k)-H}.
\end{equation*}

Let $Z\defeq M_1 M_2 \ee^H$. Theorem 2.2 (together with Remark 7.1) of \cite{renewal} proves that $\Pp(Z>t, H=S)\sim C_U t^{-\kappa}$ as $t\to\infty$, where
\begin{equation}\label{eqn:C_U}
C_U=C_I\left(\frac{C_K}{C_F}\right)^2=\kappa E[\rho_0^\kappa\log\rho_0]E[e_1](C_K)^2.
\end{equation}
The next lemma shows that the condition $\{H=S\}$ can be dropped. 

\begin{lemma}\label{lem:tail_z}
We have
\begin{equation*}
\Pp(Z>t)\sim \frac{C_U}{t^\kappa},\qquad t\to\infty. 
\end{equation*}
\end{lemma}

\begin{proof}
All moments of $M_1M_2$ are finite under $\Pp$. Indeed, $M_2\leq e_1$, $M_1\leq e_1+\sum_{k\leq 0}\ee^{-V(k)}$ and the random variables $e_1$ and $\sum_{k\leq 0}\ee^{-V(k)}$ have all moments finite under $\Pp$ (cf.~after \eqref{eqn:def_e_i} and \eqref{eqn:r-}). For any $\ell_t>0$, 
\[\Pp(Z>t, H<\ell_t)\leq\Pp(M_1 M_2>t\ee^{-\ell_t})\leq \frac{\Ep[(M_1M_2)^2]}{(t\ee^{-\ell_t})^2}.\]
Since $\kappa<2$, we may choose $\ell_t$ such that $\ell_t\to\infty$ and $t^\kappa=o(t^2\ee^{-2\ell_t})$ as $t\to\infty$, hence $\Pp(Z>t,H<\ell_t)=o(t^{-\kappa})$. On the other hand, $Z$ is independent of $S'=\sup_{x\geq e_1}V(x)-V(e_1)$, hence 
\[\Pp(Z>t, H>\ell_t, S>H)\leq\Pp(Z>t, H>\ell_t)\Pp(S'>\ell_t),\]
so that, as $t\to\infty$,
\[\Pp(Z>t, H>\ell_t) \sim \Pp(Z>t, H>\ell_t, H=S).\]
Thus we finally have 
\begin{equation*}
\Pp(Z>t)=\Pp(Z>t, H>\ell_t)+\Pp(Z>t, H<\ell_t)=C_U t^{-\kappa}+o(t^{-\kappa}).\qedhere
\end{equation*}
\end{proof}

We will actually need moments involving 
\begin{equation*}
M'_1\defeq \sum_{k<e_1}\ee^{-V(k)}
\end{equation*}
instead of $M_1(\leq M'_1)$. The next result is an adaptation of Lemma 4.1 from \cite{renewal} (together with \eqref{iglehartthm}) to the present situation, with a novelty coming from the difference between $M'_1$ and $M_1$. 

\begin{lemma}\label{lem:renewalestimates} 
For any $\alpha,\beta,\gamma\geq 0$, there is a constant $C$ such that, for large $h>0$, 
\begin{equation}\label{eqn:m1m2_inf}
\Ep[(M'_1)^\alpha(M_2)^\beta \ee^{\gamma H}|H< h]\leq \left\{\begin{array}{cl}C & \text{if }\gamma<\kappa,\\ C h&\text{if }\gamma=\kappa,\\C \ee^{(\gamma-\kappa)h}&\text{if }\gamma>\kappa\end{array}\right.
\end{equation}
and, if $\gamma<\kappa$, 
\begin{equation}\label{eqn:m1m2_sup}
\Ep[(M'_1)^\alpha(M_2)^\beta \ee^{\gamma H}|H\geq h]\leq C\ee^{\gamma h}.
\end{equation}
\end{lemma}

\begin{proof}
Let $M\defeq (M'_1)^\alpha (M_2)^\beta$. Note first that $M'_1\leq R_-+e_1$ and $M_2\leq e_1$, so that all moments of $M'_1$ and $M_2$ are finite under $\Pp$ (cf.~\eqref{eqn:r-} and after \eqref{eqn:def_e_i}). H\"older inequality then gives $\Ep[M]<\infty$ for any $\kappa$ and, since $\ee^H$ has moments up to order $\kappa$ (excluded) by \eqref{iglehartthm}, that $\Ep[(M'_1)^\alpha (M_2)^\beta \ee^{\gamma H}]<\infty$ for $\gamma<\kappa$, which proves the very first bound. 

Let us now prove in the other cases that we may insert the condition $\{S=H\}$ in the expectations. Let $\ell=\ell(h)\defeq\frac{1}{\gamma}\log h$. We have 
\begin{equation}\label{eqn:1464}
\Ep[M \ee^{\gamma H} \indic{H<h}]
	\leq \Ep[M] h+\Ep[M \ee^{\gamma H} \indic{H<h, H>\ell}]
\end{equation}
Since $M$ and $H$ are independent of $S'=\sup_{x\geq e_1} V(x)-V(e_1)$, and $\{S>H>\ell\}\subset\{S'>\ell\}$,  
\[\Ep[M \ee^{\gamma H}\indic{H<h, H>\ell}\indic{S>H}]\leq \Ep[M \ee^{\gamma H} \indic{H<h}]P(S'>\ell).\] 
Then $P(S'>\ell)\to 0$ when $h\to\infty$, hence \eqref{eqn:1464} becomes
\[\Ep[M \ee^{\gamma H} \indic{H<h}](1+o(1))\leq \Ep[M]h+\Ep[M\ee^{\gamma H} \indic{H<h}\indic{H=S}].\]
Given that $P(H=S)>0$, and $h\leq \ee^{(\gamma-\kappa)h}$ for large $h$ when $\gamma>\kappa$, it thus suffices to prove the last two bounds of \eqref{eqn:m1m2_inf} with $\Ep[M \ee^{\gamma H}|H<h,H=S]$ as the left-hand side. As for \eqref{eqn:m1m2_sup}, the introduction of $\ell$ is useless to prove similarly (skipping \eqref{eqn:1464}) that we may condition by $\{H=S\}$. 

For any $r>0$, by Lemma 4.1 of \cite{renewal}, 
\begin{equation}\label{eqn:borne_m2}
\Ep[(M_2)^r|\lfloor H\rfloor, H=S]\leq C_r.
\end{equation}

On the other hand, $M'_1=M_1+\sum_{T_H<k<e_1}\ee^{-V(k)}$. Let $r>0$. We have, by Lemma 4.1 of \cite{renewal}, $\Ep[(M_1)^r|\lfloor H\rfloor , H=S]\leq C_r$. As for the other term, it results from Lemma 3.4 of \cite{renewal} that $(H,\sum_{T_H\leq k<e_1}\ee^{-V(k)})$ has same distribution under $\Pp(\cdot|H=S)$ as $(H,\sum_{T^-_H<k\leq 0}\ee^{V(k)-H})$ where $T^-_H=\sup\{k\leq 0|V(k)>H\}$, and we claim that there is $C'_r>0$ such that, for all $N\in\N$, 
\begin{equation}\label{eqn:4646}
E\left[\bigg(\sum_{T^-_N<k\leq 0}\ee^{V(k)}\bigg)^r\right]\leq C'_r\ee^{rN}. 
\end{equation}
Before we prove this inequality, let us use it to conclude that
\begin{align}
\Ep[(M'_1)^r|\lfloor H\rfloor, H=S]
	&\leq 2^r(\Ep[(M_1)^r|\lfloor H\rfloor, H=S]+\ee^{-r\lfloor H\rfloor}C \ee^{r(\lfloor H\rfloor+1)})\notag\\
	&\leq C'. \label{eqn:borne_m1}
\end{align}
For readibility reasons, we write the proof of \eqref{eqn:4646} when $r=2$, the case of higher integer values being exactly similar and implying the general case (if $0<r<s$, $E[X^r]\leq E[X^s]^{r/s}$ for any positive $X$). 
We have
\begin{equation}\label{eqn:45678}
E\left[\bigg(\sum_{T_N^-<k\leq 0}\ee^{V(k)}\bigg)^2\right]
	\leq  \sum_{0\leq m,n<N} \ee^{n+1}\ee^{m+1}E[\nu([n,n+1)) \nu([m,m+1))]
\end{equation}
where $\nu(A)=\#\{k\leq 0:\;V(k)\in A\}$ for all $A\subset\R$. For any $n\in\N$, Markov property at time $T=\sup\{k\leq 0:\;V(k)\in[n,n+1)\}$ implies that
$E[\nu([n,n+1))^2]\leq E[\nu([-1,1))^2]$. This latter expectation is finite because $V(1)$ has a negative mean and is exponentially integrable; more precisely, $\nu([-1,1))$ is exponentially integrable as well: for $\lambda>0$, for all $k\geq 0$, $P(V(-k)<1)\leq \ee^{\lambda}E[\ee^{\lambda V(1)}]^k=\ee^\lambda E[\rho^\lambda]^k$ hence, choosing $\lambda>0$ small enough so that $E[\rho^\lambda]<1$ (cf.\ assumption ({\it a})), we have, for all $p\geq 0$, 
\begin{align*}
P(\nu([-1,1))>p)
	&\leq P(\exists k\geq p\text{ s.t. } V(-k)<1)\\
	& \leq \sum_{k\geq p}P(V(-k)<1)\leq \ee^\lambda(1-E[\rho^\lambda])^{-1} E[\rho^\lambda]^p. 
\end{align*}
Thus, using Cauchy-Schwarz inequality (or $ab\leq \frac{1}{2}(a^2+b^2)$) to bound the expectations uniformly, the right-hand side of \eqref{eqn:45678} is less than $C \ee^{2N}$ for some constant $C$. This proves \eqref{eqn:4646}. 

Finally, assembling \eqref{eqn:borne_m2} and \eqref{eqn:borne_m1} leads to
\begin{equation*}
\Ep[M|\lfloor H\rfloor, H=S]
	\leq \Ep[(M'_1)^{2\alpha}|\lfloor H\rfloor, H=S]^{1/2}\Ep[(M_2)^{2\beta}|\lfloor H\rfloor, H=S]^{1/2}
	\leq C
\end{equation*}
hence, conditioning by $\lfloor H\rfloor$, 
\begin{equation*}
\Ep[M\ee^{\gamma H}\indic{H<h}|H=S]\leq C' E[\ee^{\gamma (\lfloor H\rfloor+1)}\indic{\lfloor H\rfloor<h}]\leq C'' E[\ee^{\gamma H}\indic{H<h+1}], 
\end{equation*}
and similarly 
\begin{equation*}
\Ep[M\ee^{\gamma H}\indic{H>h}|H=S]\leq C'' E[\ee^{\gamma H}\indic{H>h-1}].
\end{equation*}
The conclusion of the lemma is then a simple consequence of the tail estimate \eqref{iglehartthm} and the usual formulas 
\begin{align*}
E[\ee^{\gamma H}\indic{H>h}] &=\ee^{\gamma h}P(H>h)+\int_h^\infty \gamma \ee^{\gamma u}P(H>u)\d u\\
\intertext{and}
E[\ee^{\gamma H}\indic{H<h}] &=1-\ee^{\gamma h}P(H>h)+\int_0^h \gamma\ee^{\gamma u}P(H>u)\d u. \qedhere
\end{align*}
\end{proof}

\subsection{Environment on the left of $\mathbf{0}$}

By definition, the distribution of the environment is translation invariant. However, the distribution of the ``environment seen from $e_k$'', i.e.~of $(\omega_{e_k+p})_{p\in\Z}$, depends on $k$. When suitably conditioning the environment on $\Z_-$, these problems vanish. 

Recall we defined both $e_i$ for $i\geq 0$ and $i\leq 0$, cf.~\eqref{eqn:def_e_negatifs}. 

\begin{lemma}\label{lem:e_negatifs}
Under $\Pp$, the sequence $(e_{i+1}-e_i)_{i\in\Z}$ is i.i.d., and more precisely the sequence of the excursions $(V(e_i+l)-V(e_i))_{0\leq l\leq e_{i+1}-e_i}$, $i\in\mathbb{Z}$, is i.i.d.. 
\end{lemma}

\begin{proof}
Let us denote $\mathcal{L}=\{\forall l<0,\ V(l)\geq 0\}$. Let $\Phi, \Psi$ be positive measurable functions respectively defined on finite paths indexed by $\{0,\ldots,k\}$ for any $k$ and on infinite paths indexed by $\mathbb{Z}_-$). 
We have 
\begin{align*}
 & \hseq\Ep[\Psi((V(e_{-1}+l)-V(e_{-1}))_{l\leq 0})\Phi((V(e_{-1}+l)-V(e_{-1}))_{0\leq l\leq -e_{-1}})]\\
	& =  \sum_{k=-\infty}^0 E[\Psi((V(k+l)-V(k))_{l\leq 0})\Phi((V(k+l)-V(k))_{0\leq l\leq -k})\ind_{A_k}]P(\mathcal{L})^{-1},
\end{align*}
where $A_k\defeq\{e_{-1}=k\}\cap\mathcal{L}=\{\forall l<k, V(l)\geq V(k), V(k)\geq V(0), \forall k<l<0, V(l)>V(k)\}$. Using the fact that $(V(k+l)-V(k))_{l\in\mathbb{Z}}$ has same distribution as $(V(l))_{l\in\mathbb{Z}}$, this becomes
\[\sum_{k=-\infty}^0 E[\Psi((V(l))_{l\leq 0})\ind_{(\forall l<0, V(l)\geq 0)}\Phi((V(l))_{0\leq l\leq -k})\ind_{(V(-k)\leq 0,\forall 0<l<-k, V(l)>0)}]P(\mathcal{L})^{-1}.\]
Finally, the independence between $(V(l))_{l\leq 0}$ and $(V(l))_{l\geq 0}$ shows that the previous expression equals 
\begin{align*}
 & \hseq\sum_{k=-\infty}^0 \Ep[\Psi((V(l))_{l\leq 0})]E[\Phi((V(l))_{0\leq l\leq -k})\ind_{(V(0)\geq V(-k),\forall 0<l<-k, V(l)>V(0))}]\\
	 & =  \Ep[\Psi((V(l))_{l\leq 0})]E[\Phi((V(l))_{0\leq l\leq e_1})],
\end{align*}
hence
\begin{align*}
  & \hseq\Ep[\Psi((V(e_{-1}+l)-V(e_{-1}))_{l\leq 0})\Phi((V(e_{-1}+l)-V(e_{-1}))_{0\leq l\leq -e_{-1}})]\\
	& =  \Ep[\Psi((V(l))_{l\leq 0})]E[\Phi((V(l))_{0\leq l\leq e_1})].
\end{align*}
By induction we deduce that, under $\Pp$, the excursions to the left are independent and distributed like the first excursion to the right. In addition, $(V(l))_{l\geq 0}$ and $(V(l))_{l\leq 0}$ are independent and, due to Markov property, the excursions to the right are i.i.d.. This concludes the proof: all the excursions, to the left or to the right, are independent and have same distribution under $\Pp$. 
\end{proof}

\section{Independence of the deep valleys} \label{sec:iid_valleys}

The independence between deep valleys goes through imposing these valleys to be disjoint (i.e. $a_i>d_{i-1}$ for all $i$) and neglecting the time spent on the left of a valley while it is being crossed (i.e.~the time spent on the left of $a_i$ before $a_{i+1}$ is reached). 

\paragraph*{NB} All the results and proofs from this section hold for any parameter $\kappa>0$. 


For any integers $x,y,z$, let us define the time
\begin{equation*}
\tautx{z}(x,y)\defeq\#\{\tau(x)\leq k\leq \tau(y):\;X_k\leq z\}
\end{equation*}
spent on the left of $z$ between the first visit to $x$ and the first visit to $y$ coming next, and the total time \[\tautx{z}\defeq\#\{k\geq \tau(z):\;X_k\leq z\}\] spent on the left of $z$ after the first visit of $z$. Of course, $\tautx{z}(x,y)\leq\tautx{z}$ if $z\leq x$. 

We consider the event
\begin{equation*}\label{def:non_overlap}
\nonoverlap\defeq\{0<a_1\}\cap\bigcap_{i=1}^{K_n-1}\{d_i<a_{i+1}\},
\end{equation*}
which means that the large valleys before $e_n$ lie entirely on $\Z_+$ and don't overlap.


The following two propositions will enable us to reduce to i.i.d.~deep valleys. 

\begin{proposition}\label{prop:overlap}
We have
\begin{equation*}\label{eq:lim_tau_overlap}
P(\nonoverlap)\limites{}{n} 1.
\end{equation*}
\end{proposition}

\begin{proof}[Proof of Proposition \ref{prop:overlap}]
Choose $\eps>0$ and define the event\[A_K(n)\defeq\{K_n\leq(1+\eps)C_I(\log n)^\kappa\}.\] Since $K_n$ is a binomial random variable of mean $nq_n\sim_n C_I(\log n)^\kappa$, it follows from the law of large numbers that $P(A_K(n))$ converges to $1$ as $n\to\infty$. On the other hand, if the event $\overlap$ occurs, then there exists $1\leq i\leq K_n$ such that there is at least one high excursion among the first $D_n$ excursions to the right of $d_{i-1}$ (with $d_0=0$). Thus, 
\begin{align*}
P(\overlap)
	& \leq P(A_K(n)^c)+(1+\eps)C_I(\log n)^\kappa (1-(1-q_n)^{D_n})\\
	& \leq o(1)+(1+\eps)C_I(\log n)^\kappa q_n D_n = o(1)
\end{align*}
Indeed, for any $0<x<1$ and $\alpha>0$, we have $1-(1-x)^\alpha\leq\alpha x$ by concavity of $x\mapsto 1-(1-x)^\alpha$.
\end{proof}

For $x\geq 0$, define
\begin{equation}
a(x)\defeq \max\{e_k:\;k\in\Z,\ e_{k+D_n}\leq x\}. 
\end{equation}
In particular, $a(b_i)=a_i$ for all $i\geq 1$, and $a(0)=e_{-D_n}$. 

\begin{proposition}\label{prop:lim_taut}
Under $\bP$, 
\begin{equation*}\label{eq:lim_taut_a}
\frac{1}{n^{1/\kappa}}\sum_{i=1}^{K_n}\tautx{a_i}(b_i,d_i)=\frac{1}{n^{1/\kappa}}\sum_{k=0}^{n-1}\tautx{a(e_k)}(e_k,e_{k+1})\indic{H_k\geq h_n}\limites{(p)}{n} 0.
\end{equation*}
\end{proposition}



\begin{proof}
The equality is trivial from the definitions. The second expression has the advantage that, under $\bPp$, all the terms have same distribution because of Lemma~\ref{lem:e_negatifs}. To overcome the fact that $\tautx{a(0)}(0,e_1)\indic{H\geq h_n}$ is not integrable for $0<\kappa\leq 1$, we introduce the event 
\begin{equation*}
	\begin{split}
A_n
	&\defeq \{\text{for }i=1,\ldots,K_n,\ H_{\sigma(i)}\leq V(a_i)-V(b_i)\}\\
	& = \bigcap_{k=0}^{n-1}\{H_k<h_n\}\cup\{h_n\leq H_k\leq V(e_{k-D_n})-V(e_k)\}.
	\end{split} 
\end{equation*}
Let us prove that $\Pp((A_n)^c)=o_n(1)$. By Lemma~\ref{lem:e_negatifs}, 
\begin{align*}
\Pp((A_n)^c)
	& \leq n \Pp(H\geq h_n,\ H> V(e_{-D_n})). 
\end{align*}
Choose $0<\gamma'<\gamma''<\gamma$ (cf.~\eqref{eqn:def_d_n}) and define $l_n=\frac{1+\gamma'}{\kappa}\log n$. We get
\begin{align}
\Pp((A_n)^c)
	 \leq n\big( P(H\geq l_n)+P(H\geq h_n)\Pp(V(e_{-D_n})<l_n)\big). \label{eqn:4697}
\end{align} 
Equation \eqref{iglehartthm} gives $P(H\geq l_n)\sim_n C_I\ee^{-\kappa l_n}=C_I n^{-(1+\gamma')}$ and $P(H\geq h_n)\sim_n C_I n^{-1}(\log n)^\kappa$. Under $\Pp$, $V(e_{-D_n})$ is the sum of $D_n$ i.i.d.~random variables distributed like $-V(e_1)$. Therefore, for any $\lambda>0$, 
\begin{equation*}
\Pp(V(e_{-D_n})<l_n)
	\leq \ee^{\lambda l_n} E[\ee^{-\lambda (-V(e_1))}]^{D_n}. 
\end{equation*}
Since $\frac{1}{\lambda}\log E[\ee^{-\lambda(-V(e_1))}]\to -E[-V(e_1)]\in[-\infty,0)$ as $\lambda\to 0^+$, we can choose $\lambda>0$ such that $\log E[\ee^{-\lambda(-V(e_1))}]<-\lambda A\frac{1+\gamma''}{1+\gamma}$ (where $A$ was defined after \eqref{eqn:def_d_n}), hence $E[\ee^{-\lambda(-V(e_1))}]^{D_n}\leq n^{-\lambda\frac{1+\gamma''}{\kappa}}$. Thus, $\Pp(V(e_{-D_n})<l_n)\leq n^{-\lambda \frac{\gamma''-\gamma'}{\kappa}}$. Using these estimates in \eqref{eqn:4697} concludes the proof that $\Pp((A_n)^c)=o_n(1)$. 

Let us now prove the Proposition itself. By Markov inequality, for all $\delta>0$, 
\begin{align}
 &  \hseq\bPp\left(\frac{1}{n^{1/\kappa}}\sum_{k=0}^{n-1}\tautx{a(e_k)}(e_k,e_{k+1})\indic{H\geq h_n}>\delta\right)\notag\\
	 & \leq  \Pp((A_n)^c)+\frac{1}{\delta n^{1/\kappa}}\bEp\left[\sum_{k=0}^{n-1}\tautx{a(e_k)}(e_k,e_{k+1})\indic{H\geq h_n}\ind_{A_n}\right]\notag\\
	 & \leq  o_n(1) + \frac{n}{\delta n^{1/\kappa}}\Ep\left[E_\omega[\tautx{e_{-D_n}}(0,e_1)]\indic{H>h_n,\ H<V(e_{-D_n})}\right]. \label{eqn:1564}
\end{align}
Note that we have $E_\omega[\taut^{(e_{-D_n})}(0,e_1)]=E_\omega[N]E_\omega[T_1]$, where $N$ is the number of crossings from $e_{-D_n}+1$ to $e_{-D_n}$ before the first visit at $e_1$, and $T_1$ is the time for the random walk to go from $e_{-D_n}$ to $e_{-D_n}+1$ (for the first time, for instance); furthermore, these two terms are independent under $\Pp$. Using \eqref{eqn:zeitouni_p}, we have
\begin{equation*}
E_\omega[N]=\frac{P_{0,\omega}(\tau(e_{-D_n})<\tau(e_1))}{P_{e_{-D_n}+1,\omega}(\tau(e_1)<\tau(e_{-D_n}))}= \sum_{0\leq x<e_1}\ee^{V(x)-V(e_{-D_n})}=M_2\ee^{H-V(e_{-D_n})}
\end{equation*}
hence, on the event $\{H<V(e_{-D_n})\}$, $E_\omega[N]\leq M_2$. 

The length of an excursion to the left of $e_{-D_n}$ is computed as follows, due to \eqref{eqn:zeitouni_e3}:
\begin{equation*}
E_\omega[T_1]=E_{e_{-D_n},\omega}[\tau(e_{-D_n}+1)]\leq 2\sum_{x\leq e_{-D_n}}\ee^{-(V(x)-V(e_{-D_n}))}. 
\end{equation*}
The law of $(V(x)-V(e_{-D_n}))_{x\leq e_{-D_n}}$ under $\Pp$ is $\Pp$ because of Lemma \ref{lem:e_negatifs}. Therefore, 
\begin{equation*}
\Ep[E_\omega[T_1]]\leq 2\Ep\bigg[\sum_{x\leq 0}\ee^{-V(x)}\bigg]=2\Ep[R_-]<\infty
\end{equation*}
with \eqref{eqn:r-}. We conclude that the right-hand side of \eqref{eqn:1564} is less than
\begin{align*}
o_n(1)+\frac{n}{\delta n^{1/\kappa}}2\Ep[R_-]E[M_2\indic{H>h_n}]. 
\end{align*}
Since Lemma \ref{lem:renewalestimates} gives $E[M_2\indic{H>h_n}]\leq CP(H>h_n)\sim_n C'\ee^{-\kappa h_n}=C'n^{-1}(\log n)^\kappa$, this whole expression converges to 0, which concludes. 
\end{proof}

\section{Fluctuation of interarrival times} \label{sec:interarrival}

For any $x\leq y$, we define the inter-arrival time  $\tau(x,y)$ between sites $x$ and $y$ by 
\begin{equation*}\label{hittimerwia}
    \tau(x,y)\defeq  \inf \{ n \ge 0: \, X_{\tau(x)+n}=y \}, \qquad x,y \in \Z.
\end{equation*}

Let
\begin{equation*}
\tia\defeq\sum_{i=0}^{K_n}\tau(d_i,b_{i+1}\wedge e_n)=\sum_{k=0}^{n-1}\tau(e_k,e_{k+1})\indic{H_k<h_n}
\end{equation*}
(with $d_0=0$) be the time spent at crossing small excursions before $\tau(e_n)$. The aim of this section is the following bound on the fluctuations of $\tia$. 
\begin{proposition}\label{prop:tau_ia}
For any $1\leq \kappa<2$, under $\bPp$, 
\begin{equation} \label{eqn:limite_fluctu_tia}
\frac{1}{n^{1/\kappa}}(\tia-\bEp[\tia])\limites{(p)}{n}0. 
\end{equation}
\end{proposition}

\begin{remark*}
This limit also holds for $0<\kappa<1$ in a simple way: we have, in this case, using Lemma \ref{lem:renewalestimates}, 
\begin{align*}
\bEp[\tia]=n\Ep[E_\omega[\tau(e_1)]\indic{H\leq h_n}]
	&\leq nE[2M'_1M_2\ee^H\indic{H\leq h_n}]\notag\\
	&\leq C n\ee^{(1-\kappa)h_n}=o(n^{1/\kappa}),
\end{align*}
hence $n^{-1/\kappa}\tia$ and its expectation separately converge to 0 in probability.
\end{remark*}

By Chebychev inequality, this Proposition will come as a direct consequence of Lemma~\ref{lem:var_tia} bounding the variance of $\tia$. However, a specific caution is necessary in the case $\kappa=1$; indeed, the variance is infinite in this case, because of the rare but significant fluctuations originating from the time spent by the walk when it backtracks into deep valleys. Our proof in this case consists in proving first that we may neglect in probability (using a first-moment method) the time spent backtracking into deep valleys; and then that this brings us to the problem of the computation of the variance of $\tia$ in an environment where small excursions have been substituted for the high ones (thus removing the non-integrability problem). 

Subsection \ref{subsec:reduc_small} is dedicated to this reduction to an integrable setting, which is only involved in the case $\kappa=1$ of Proposition \ref{prop:tau_ia} and of the main theorem (but holds in greater generality), while Subsection \ref{subsec:var_tia} states and proves the bounds on the variance, both for the initial (for $\kappa>1$) and the modified environment, implying Proposition~\ref{prop:tau_ia}. 

\subsection{Reduction to small excursions (required for the case $\kappa=1$)}\label{subsec:reduc_small}

Let $h>0$. Let us denote by $d_-$ the right end of the first excursion on the left of 0 that is higher than $h$: 
\begin{equation*}
d_- \defeq \max\{e_k:\;k\leq 0, H_{k-1}\geq h\}. 
\end{equation*}
Remember $\tautx{d_-}(0,e_1)$ is the time spent on the left of $d_-$ before the walk reaches $e_1$. 

\begin{lemma}\label{lem:remontee_difficile}
There exists $C>0$, independent of $h$, such that
\begin{equation}\label{eqn:time_left_d}
\bEp[\tautx{d_-}(0,e_1)\indic{H\leq h}]\leq C \left\{
\begin{array}{cl}
	\ee^{-(2\kappa-1)h}&\mbox{if }\kappa<1\\
	h\ee^{- h}&\mbox{if }\kappa=1\\
	\ee^{-\kappa h}&\mbox{if }\kappa>1.
\end{array}\right.
\end{equation}
\end{lemma}


\begin{proof}
Let us decompose $\tautx{d_-}(0,e_1)$ into the successive excursions to the left of $d_-$: 
\begin{equation*}
\tautx{d_-}(0,e_1)=\sum_{m=1}^N T_m,
\end{equation*}
where $N$ is the number of crossings from $d_-+1$ to $d_-$ before $\tau(e_1)$, and $T_m$ is the time for the walk to go from $d_-$ to $d_-+1$ on the $m$-th time. Under $P_\omega$, the times $T_m$, $m\geq 1$, are i.i.d.\ and independent of $N$ (i.e., more properly, the sequence $(T_m)_{1\leq m\leq N}$ can be prolonged to an infinite sequence with these properties). We have, using Markov property and then \eqref{eqn:zeitouni_p},
\begin{equation*}
E_\omega[N]=\frac{P_{0,\omega}(\tau(d_-)<\tau(e_1))}{P_{d_-+1,\omega}(\tau(e_1)<\tau(d_-))}=\sum_{0\leq x<e_1} \ee^{V(x)-V(d_-)}
\end{equation*}
and, from \eqref{eqn:zeitouni_e3},
\begin{equation*}
E_\omega[T_1]=E_{d_-,\omega}[\tau(d_-+1)]\leq 2\sum_{x\leq d_-}\ee^{-(V(x)-V(d_-))}. 
\end{equation*}
Therefore, by Wald identity and Lemma \ref{lem:e_negatifs}, 
\begin{align}
\bEp[\tautx{d_-}(0,e_1)\indic{H\leq h}]
	& = \Ep[E_\omega[N]E_\omega[T_1]\indic{H\leq h}]\notag\\
	& \hseq\leq 2 E\bigg[\sum_{0\leq x<e_1}\ee^{V(x)}\indic{H\leq h}\bigg]
		\Ep[\ee^{-V(d_-)}]
		E^{\Lambda(h)}\bigg[\sum_{x\leq 0}\ee^{-V(x)}\bigg],\label{eqn:7878}
\end{align}
where $\Lambda(h)\defeq\{\forall k\leq 0,\ V(k)\geq 0\}\cap\{H_{-1}\geq h\}$. The first expectation can be written as $E[M_2\ee^H \indic{H\leq h}]$. For the second one, note that $d_-=e_{-W}$, where $W$ is a geometric random variable of parameter $q\defeq P(H\geq h)$; and, conditional on $\{W=n\}$, the distribution of $(V(k))_{e_{-W}\leq k\leq 0}$ under $\Pp$ is the same as that of $(V(k))_{e_{-n}\leq k\leq 0}$ under $\Pp(\cdot|\text{for }k=0,\ldots,n-1, H_{-k}< h)$. Therefore, 
\begin{align*}
\Ep[\ee^{-V(d_-)}]
	& = \Ep\big[E[\ee^{-V(e_1)}|H<h]^W\big] = \frac{q}{1-(1-q)E[\ee^{V(e_1)}|H<h]},
\end{align*}
and $(1-q)E[\ee^{V(e_1)}|H<h]$ converges to $E[\ee^{V(e_1)}]<1$ when $h\to\infty$ (the inequality comes from assumption ({\it b})), hence this quantity is uniformly bounded from above by $c<1$ for large $h$. In addition, \eqref{iglehartthm} gives $q\sim C_I\ee^{-\kappa h}$ when $h\to\infty$, hence
\begin{equation*}
\Ep[\ee^{-V(d_-)}]\leq C\ee^{-\kappa h},
\end{equation*}
where $C$ is independent of $h$. Finally, let us consider the last term of \eqref{eqn:7878}. We have
\begin{align*}
&\hspace{-.5cm} E^{\Lambda(h)}\bigg[\sum_{x\leq 0}\ee^{-V(x)}\bigg]\\
	& = E\bigg[\sum_{e_{-1}<x\leq 0}\ee^{-V(x)}\bigg|H_{-1}\geq h\bigg]
		+\Ep\bigg[\sum_{x\leq e_{-1}}\ee^{-(V(x)-V(e_{-1}))}\bigg]E[\ee^{-V(e_{-1})}]\\
	& \leq E[M'_1|H\geq h]+\Ep[R_-]E[\ee^{V(e_1)}]
\end{align*}
hence, using Lemma \ref{lem:renewalestimates}, \eqref{eqn:r-} and $V(e_1)\leq 0$, this term is bounded by a constant. The statement of the lemma then follows from the application of Lemma \ref{lem:renewalestimates} to $E[M_2\ee^H\indic{H\leq h}]$. 
\end{proof}

The part of the inter-arrival time $\tia$ spent at backtracking in high excursions can be written as follows: 
\begin{equation*}
	\begin{split}
\widetilde{\tia}
	&\defeq \tautx{d_-}(0,b_1\wedge e_n)+\sum_{i=1}^{K_n}\tautx{d_i}(d_i,b_{i+1}\wedge e_n)\\
	&=\sum_{k=0}^{n-1} \tautx{d(e_k)}(e_k,e_{k+1})\indic{H_k<h_n},
	\end{split}
\end{equation*}
where, for $x\in\Z$,
\begin{equation*}
d(x)\defeq\max\{e_k:\;k\in\Z,\ e_k\leq x,\ H_{k-1}\geq h_n\}. 
\end{equation*}
In particular, $d(0)=d_-$ in the previous notation with $h=h_n$. 

Note that, under $\bPp$, because of Lemma \ref{lem:e_negatifs}, the terms of the above sum have same distribution as $\tautx{d(0)}(0,e_1)\indic{H< h_n}$, hence
\begin{equation*}
\bEp[\widetilde{\tia}]= n\bEp[\tautx{d(0)}(0,e_1)\indic{H< h_n}]. 
\end{equation*}
Thus, for $\bEp[\widetilde{\tia}]$ to be negligible with respect to $n^{1/\kappa}$, it suffices that the expectation on the right-hand side be negligible with respect to $n^{1/\kappa-1}$. In particular, for $\kappa=1$, it suffices that it converges to 0, which is readily seen from \eqref{eqn:time_left_d}. Thus, for $\kappa=1$, 
\begin{equation}  \label{eqn:lim_tia_tilde}
\frac{1}{n^{1/\kappa}}\bEp[\widetilde{\tia}]\limites{}{n}0, 
\end{equation}
hence in particular $n^{-1/\kappa}\widetilde{\tia}\to 0$ in probability under $\bPp$. Note that \eqref{eqn:lim_tia_tilde} actually holds for any $\kappa\geq 1$. 

Let us introduce the modified environment, where independent small excursions are substituted for the high excursions. In order to avoid obfuscating the redaction, we will only introduce little notation regarding this new environment. 

Let us enlarge the probability space in order to accommodate for a new family of independent excursions indexed by $\N^*\times\Z$ such that the excursion indexed by $(n,k)$ has same distribution as $(V(x))_{0\leq x\leq e_1}$ under $P(\cdot|H\leq h_n)$. Thus we are given, for every $n\in\N^*$, a countable family of independent excursions lower than $h_n$. For every $n$, we define the modified environment of height less than $h_n$ by replacing all the excursions of $V$ that are higher than $h_n$ by new independent ones that are lower than $h_n$. Because of Lemma \ref{lem:e_negatifs}, this construction is especially natural under $\Pp$, where it has stationarity properties. 

In the following, we will denote by $P'$ the law of the modified environment relative to the height $h_n$ guessed from the context (hence also a definition of $(\Pp)'$, for instance). 

\begin{remark*}\label{rem:e'r}
Repeating the proof done under $\Pp$ for $(\Pp)'$, we see that $R_-$ still has all finite moments in the modified environment, and that these moments are bounded uniformly in $n$. In particular, the bound for $\Ep[(M'_1)^\alpha (M_2)^\beta \ee^{\gamma H}\indic{H\leq h_n}]$ given in Lemma \ref{lem:renewalestimates} is unchanged for $(\Ep)'$ (writing $M'_1=R_-+\sum_{0\leq k<e_1}\ee^{-V(x)}$ and using $(a+b)^\alpha\leq 2^\alpha(a^\alpha+b^\alpha)$). On the other hand, 
\begin{align*}
E'[R]
	&= \sum_{i=0}^\infty E'[\ee^{V(e_i)}]E'\left[\sum_{e_i\leq k<e_{i+1}}\ee^{V(k)-V(e_i)}\right]\\
	&= \sum_{i=0}^\infty E[\ee^{V(e_1)}|H\leq h_n]^i E[M_2\ee^H|H\leq h_n],
\end{align*}
and $E[\ee^{V(e_1)}|H\leq h_n]\leq c$ for some $c<1$ independent of $n$ because this expectation is smaller than 1 for all $n$ and it converges toward $E[\ee^{V(e_1)}]<1$ as $n\to\infty$. Hence, by Lemma \ref{lem:renewalestimates}, 
\begin{equation}\label{eqn:e'r}
\text{if $\kappa=1$},\qquad E'[R] \leq C h_n 
\end{equation}
This is the only difference that will appear in the following computations. 
\end{remark*}

Assuming that $d(0)$ keeps being defined with respect to the usual heights, \eqref{eqn:time_left_d} (with $h=h_n$) is still true for the walk in the modified environment. Indeed, the change only affects the environment on the left of $d(0)$, hence the only difference in the proof involves the times $T_m$: in \eqref{eqn:7878}, one should substitute $(\Ep)'$ for $E^{\Lambda(h)}$, and this factor is uniformly bounded in both cases because of the above remark about $R_-$. 

We deduce that the time $\widetilde{\tia}'$, defined similarly to $\widetilde{\tia}$ except that the excursions on the left of the points $d(e_i)$ (i.e.\ the times similar to $T_m$ in the previous proof) are performed in the modified environment, still satisfies: for $\kappa=1$, 
\begin{equation} \label{eqn:lim_tia_tilde'}
\frac{1}{n^{1/\kappa}}\bEp[\widetilde{\tia}']\limites{}{n}0. 
\end{equation}

Note now that
\begin{equation*}
\tia'\defeq\tia-\widetilde{\tia}+\widetilde{\tia}'
\end{equation*}
is the time spent at crossing the (original) small excursions, in the environment where the high excursions have been replaced by new independent small excursions. Indeed, the high excursions are only involved in $\tia$ during the backtracking of the walk to the left of $d(e_i)$ for some $0\leq i< n$. Assembling \eqref{eqn:lim_tia_tilde} and \eqref{eqn:lim_tia_tilde'}, it is equivalent (for $\kappa=1$) to prove \eqref{eqn:limite_fluctu_tia} or 
\begin{equation*}
\frac{1}{n^{1/\kappa}}(\tia'-\bEp[\tia'])\limites{(p)}{n}0, 
\end{equation*}
and it is thus sufficient to prove 
\begin{equation*}
\bVarp(\tia')=o_n(n^{2/\kappa}). 
\end{equation*}

\subsection{Bounding the variance of $\tia$} \label{subsec:var_tia}

Because of the previous subsection, the proof of Proposition \ref{prop:tau_ia} will follow from the following Lemma. 

\begin{lemma}\label{lem:var_tia}
We have, for $1<\kappa<2$, 
\begin{equation*}
\bVarp(\tia)=o_n(n^{2/\kappa})
\end{equation*}
and, for $1\leq \kappa<2$, 
\begin{equation*}
\bVarp(\tia')=o_n(n^{2/\kappa}). 
\end{equation*}
\end{lemma}

We recall that the second bound is only introduced to settle the case $\kappa=1$; it would suffice for $1<\kappa<2$ as well, but introduces unnecessary complication. The computations being very close for $\tia$ and $\tia'$, we will write below the proof for $\tia$ and indicate line by line where changes happen for $\tia'$. Let us stress that, when dealing with $\tia'$, all the indicator functions $\indic{H_\cdot\leq h_n}$ (which define the small valleys) would refer to the original heights, while all the potentials $V(\cdot)$ appearing along the computation (which come from quenched expectations of times spent by the walk) would refer to the modified environment. 

Since we have
\begin{equation*}
\bVarp(\tia)=\Ep[\Var_\omega(\tia)]+\Varp(E_\omega[\tia]), 
\end{equation*}
it suffices to prove the following two results:
\begin{align}
\Ep[\Var_\omega(\tia)]&=o_n(n^{2/\kappa}) \label{eqn:lim_e_var}\\
\Varp(E_\omega[\tia]) &=o_n(n^{2/\kappa}) \label{eqn:lim_var_e}.
\end{align}

\subsubsection{Proof of \eqref{eqn:lim_e_var}}

We have 
\begin{equation}\label{eqn:tau_ia_excursions}
\tau_{IA}=\sum_{p=0}^{n-1} \tau(e_p,e_{p+1})\indic{H_p< h_n}
\end{equation}
and by Markov property, the above times are independent under $P_{o,\omega}$. Hence 
\begin{equation*}
\Var_\omega(\tia)=\sum_{p=0}^{n-1}\Var_\omega(\tau(e_p,e_{p+1}))\indic{H_p< h_n}.
\end{equation*}
Under $\Pp$, the distribution of the environment seen from $e_p$ does not depend on $p$, hence 
\begin{equation} \label{eq:var_1}
\Ep[\Var_\omega(\tia)]=n \Ep[\Var_\omega(\tau(e_1))\indic{H< h_n}].
\end{equation}

We use Formula \eqref{eqn:alili2}: 
\begin{equation} \label{eq:var_zeitouni_1}
\Var_\omega(\tau(e_1))\indic{H< h_n}\leq 16\sum_{l'\leq l\leq j\leq k\leq e_1,\ 0\leq k} \ee^{{V}(k)+{V}(j)-{V}(l)-{V}(l')}\indic{H< h_n}.
\end{equation}
Let us first consider the part of the sum where $j\geq 0$. By noting that the indices satisfy $l'\leq j$ and $l\leq k$, this part is seen to be less than $(M'_1 M_2\ee^H)^2\indic{H<h_n}$. Lemma \ref{lem:renewalestimates} shows that its expectation is less than $C\ee^{(2-\kappa)h_n}$. \emph{For $\tia'$: The same holds, because of the remark p.~\pageref{rem:e'r}. }

It remains to deal with the indices $j<0$. This part rewrites as
\begin{equation}\label{eqn:j_neg}
\sum_{l',l\leq j< 0}\ee^{V(j)-V(l)-V(l')}\cdot\sum_{0\leq k<e_1}\ee^{V(k)}\indic{H<h_n}. 
\end{equation}
Since $V_{|\Z_+}$ and $V_{|\Z_-}$ are independent under $P$, so are the two above factors. The second one equals $\ee^H M_2\indic{H<h_n}$. Let us split the first one according to the excursion $[e_{u-1},e_u)$ containing $j$; it becomes 
\begin{equation}\label{eqn:78648}
\sum_{u\leq 0}\ee^{-V(e_{u-1})}\sum_{e_{u-1}\leq j<e_u}\ee^{V(j)-V(e_{u-1})}\left(\sum_{l\leq j} \ee^{-(V(l)-V(e_{u-1}))}\right)^2.
\end{equation}
We have $V(e_{u-1})\geq V(e_u)$ and, under $\Pp$, $V(e_u)$ is independent of $(V(e_u+k)-V(e_{u}))_{k\leq 0}$ and thus of $(V(e_{u-1}+k)-V(e_{u-1}))_{k\leq e_{u}-e_{u-1}}$, which has same distribution as $(V(k))_{k\leq e_1}$. Therefore, the expectation of \eqref{eqn:78648} with respect to $\Pp$ is less than
\begin{align*}
&\sum_{u\leq 0} \Ep[\ee^{-V(e_u)}]\Ep\left[\sum_{0\leq j<e_1}\ee^{V(j)}\left(\sum_{l\leq j}\ee^{-V(l)}\right)^2\right]\\
&\hseq\leq (1-E[\ee^{V(e_1)}])^{-1}\Ep[\ee^H(M'_1)^2M_2].
\end{align*}
Thus the expectation of \eqref{eqn:j_neg} with respect to $\Pp$ is bounded by 
\begin{equation*}
(1-E[\ee^{V(e_1)}])^{-1}\Ep[\ee^H(M'_1)^2M_2]\Ep[\ee^H M_2\indic{H<h_n}]. 
\end{equation*}
From Lemma \ref{lem:renewalestimates}, we conclude that this term is less than a constant if $\kappa>1$. The part corresponding to $j\geq 0$ therefore dominates; this finishes the proof of \eqref{eqn:lim_e_var}. \emph{For $\tia'$: The first factor is $(1-E[\ee^{V(e_1)}|H<h_n])^{-1}$, which is uniformly bounded because it converges to $(1-E[\ee^{V(e_1)}])^{-1}<\infty$ and, using Lemma \ref{lem:renewalestimates}, the two other factors are each bounded by a constant if $\kappa>1$ and by $C h_n$ if $\kappa=1$ (cf.~again the remark p.~\pageref{rem:e'r}). Thus, the part corresponding to $j\geq 0$ still dominates in this case.}

We have proved $\Ep[\Var_\omega(\tia))]\leq C n\ee^{(2-\kappa)h_n}$. Since $n\ee^{(2-\kappa)h_n}=\frac{n^{2/\kappa}}{(\log n)^{2-\kappa}}$, this concludes.

\subsubsection{Proof of \eqref{eqn:lim_var_e}}

From equation~\eqref{eqn:tau_ia_excursions} we deduce
\begin{equation*}
E_\omega[\tia]=\sum_{p=0}^{n-1} E_\omega[\tau(e_p,e_{p+1})]\indic{H_p< h_n}
\end{equation*}
hence, using \eqref{eqn:zeitouni_e}, 
\begin{equation}\label{eqn:var_123}
\Varp(E_\omega[\tia])
	= \sum_{i\leq j,\ k\leq l}(\Ep[A_{ij}A_{kl}]-\Ep[A_{ij}]\Ep[A_{kl}]),
\end{equation}
where $A_{ij}\defeq\alpha_{ij}\ee^{V(j)-V(i)}\indic{0\leq j<e_n,\ H(j)<h_n}$ for any indices $i<j$, and $H(j)=H_q$ when $e_q\leq j<e_{q+1}$. 

Let us split this sum according to the relative order of $i,j,k,l$ and bound each term separately. Note that, up to multiplication by $2$, we may assume $j\leq l$, hence we only have to consider $i\leq j\leq k\leq l$ and $i,k\leq j\leq l$ (either $i\leq k$ or $k\leq i$). 

$\bullet$ $i\leq j\leq k\leq l$. Let us split again according to the excursion containing $j$. The summand of~\eqref{eqn:var_123} equals
\begin{equation}\label{eqn:split_excursion}
\sum_{q=0}^{n-1}(\Ep[A_{ij}\indic{e_q\leq j<e_{q+1}}A_{kl}]-\Ep[A_{ij}\indic{e_q\leq j<e_{q+1}}]\Ep[A_{kl}]).
\end{equation}
In addition, we write $A_{kl}=A_{kl}\indic{k\geq e_{q+1}}+A_{kl}\indic{e_q\leq k<e_{q+1}}$ in both terms in order to split according to whether $j$ and $k$ lie in the same excursion. 

Let us consider the case when $j$ and $k$ are in different excursions. Because of Markov property at time $e_{q+1}$ and of the stationarity of the distribution of the environment, we have, for any $i\leq j\leq k\leq l$,
\begin{equation*}
\Ep[A_{ij}\indic{e_q\leq j<e_{q+1}}A_{kl}\indic{e_{q+1}\leq k}]
	= \Ep[A_{ij}\indic{e_q\leq j<e_{q+1}}]\Ep[A_{kl}\indic{e_{q+1}\leq k}], 
\end{equation*}
hence these terms do not contribute to the sum \eqref{eqn:split_excursion}. \emph{The same holds for $\tia'$.}

For the remainder of the sum, we will only need to bound the ``expectation of the square'' part of the variance, i.e.~the terms coming from $\Ep[A_{ij}A_{kl}]$. 

Let us turn to the case when $j$ and $k$ lie in the same excursion $[e_q,e_{q+1})$. We have (remember $\alpha_{ij}\leq 2$)
\begin{align*}
 & \hseq \sum_{i\leq j\leq k\leq l}\sum_{q=0}^{n-1} \Ep\left[A_{ij}A_{kl} \indic{e_q\leq j\leq k<e_{q+1}}\right]\\
	& \leq 4\sum_{q=0}^{n-1}\Ep\left[\sum_{i\leq j\leq k\leq l}\ee^{V(j)-V(i)}\indic{e_q\leq j\leq k<e_{q+1}}\ee^{V(l)-V(k)}\indic{H_q<h_n}\right].
\end{align*}
Because of Lemma~\ref{lem:e_negatifs}, the last expectation, which involves a function of $(V(e_q+l)-V(e_q))_{l\in\Z}$, does not depend on $q$. Thus it equals 
\begin{equation*}
	\begin{split}
&\hseq 4n\sum_{i\leq j\leq k\leq l}\Ep[\ee^{V(j)-V(i)}\indic{0\leq j\leq k<e_1}\ee^{V(l)-V(k)}\indic{H<h_n}]\\
& \leq 4n \Ep\left[\sum_{i\leq j\leq e_1,\ 0\leq j}\ee^{V(j)-V(i)}\sum_{0\leq k\leq l,\ k< e_1}\ee^{V(l)-V(k)}\indic{H<h_n}\right].
	\end{split}
\end{equation*}
Splitting according to whether $l<e_1$ or $l\geq e_1$, the variable in the last expectation is bounded by $(M'_1M_2\ee^{H})^2\indic{H<h_n}+(M'_1M_2\ee^H)(M'_1\sum_{l\geq e_1}\ee^{V(l)})\indic{H<h_n}$. Note that $\sum_{l\geq e_1}\ee^{V(l)}\leq\sum_{l\geq e_1}\ee^{V(l)-V(e_1)}$, which has same distribution as $R$ and is independent of $M'_1,M_2,H$. The above bound thus becomes 
\begin{equation*}
4n(\Ep[(M'_1)^2(M_2)^2\ee^{2H}\indic{H<h_n}]+\Ep[(M'_1)^2M_2\ee^H\indic{H<h_n}]E[R]). 
\end{equation*}
From Lemma \ref{lem:renewalestimates}, this is less than $4n(C\ee^{(2-\kappa)h_n}+C)\leq C'n\ee^{(2-\kappa)h_n}$. \emph{For $\tia'$, this is unchanged when $\kappa>1$; and, if $\kappa=1$, the above expression is bounded by $4n(C\ee^{(2-\kappa)h_n}+C(h_n)^2)$, cf.~\eqref{eqn:e'r}, hence the bound remains the same. }

$\bullet$ $i,k\leq j\leq l$ (either $i\leq k$ or $k\leq i$). We have 
\begin{equation*}
	\begin{split}
&  \hseq \sum_{i, k\leq j\leq l}\Ep[A_{ij}A_{kl}]\\
	& \leq  8\sum_{p=0}^{n-1} \Ep\left[\sum_{i\leq k\leq j\leq l}\ee^{V(l)-V(k)+V(j)-V(i)}\indic{e_p\leq l<e_{p+1}}\indic{H_p<h_n}\right]. 
	\end{split}
\end{equation*}
Using Lemma~\ref{lem:e_negatifs} we see that the above expectation, which involves a function of $(V(e_p+l)-V(e_p))_{l\in\Z}$, does not depend on $p$. Therefore, it equals 
\begin{equation*}
8 n \Ep\left[\sum_{i\leq k\leq j\leq l\leq e_1,\ l\geq 0}\ee^{V(l)+V(j)-V(k)-V(i)}\indic{H<h_n}\right]. 
\end{equation*}
The quantity in the expectation matches exactly the formula in (\ref {eq:var_zeitouni_1}) that was used as a bound for $\Var_\omega(\tau(e_1))\indic{H< h_n}$ (with different names for the indices: $(i,k,j,l)$ becomes $(l',l,j,k)$). Thus, it follows from the proof of \eqref{eqn:lim_e_var} that 
\begin{equation*}
\sum_{i, k\leq j\leq l}\Ep[A_{ij}A_{kl}] \leq C n \ee^{(2-\kappa)h_n}=o_n(n^{2/\kappa}). 
\end{equation*}

We have obtained the expected upper bound for each of the orderings, hence the lemma. 

\subsection{A subsequent Lemma}

The previous proofs of \eqref{eqn:lim_e_var} and \eqref{eqn:lim_var_e} entail the following bound for the crossing time of one low excursion: 
\begin{lemma}\label{lem:tau_e_1}
We have, for $1<\kappa<2$: 
\begin{equation*}
\Ep[E_\omega[\tau(e_1)^2]\indic{H<h}]\leq C \ee^{(2-\kappa)h}, 
\end{equation*}
and similarly for $(\Ep)'$ when $1\leq \kappa<2$. 
\end{lemma}

\begin{proof}
We have $E_\omega[\tau(e_1)^2]=\Var_\omega(\tau(e_1))+E_\omega[\tau(e_1)]^2$. Equation~\eqref{eq:var_1} and the remainder of the proof of \eqref{eqn:lim_e_var} give: 
\begin{equation*}
\Ep[\Var_\omega(\tau(e_1))\indic{H< h}]\leq C\ee^{(2-\kappa)h}. 
\end{equation*}
In order to see that the proof of \eqref{eqn:lim_var_e} implies the remainding bound: 
\begin{equation*}
\Ep[E_\omega[\tau(e_1)]^2\indic{H< h}]\leq C\ee^{(2-\kappa)h}, 
\end{equation*}
it suffices to take $n=1$ in the proof (except of course in ``$h_n$'') and to notice that, although our proof gave a bound for the \emph{variance} of $E_\omega[\tau(e_1)]$, we actually only needed to substract the ``squared expectation''-terms (cf.\ \eqref{eqn:var_123}) corresponding to indices lying in different excursions\ldots\ a situation which doesn't occur when $n=1$. Thus our proof in fact gives (in this case only) a bound for the ``expectation of the square'' of $E_\omega[\tau(e_1)]$. 
\end{proof}

\section{A general estimate for the occupation time of a deep valley} \label{sec:propplougon}

In this section we establish a precise annealed estimate for the tail distribution of the time spent by the particle to cross the first positive excursion
of the potential above its past infimum. Since we shall use this result to estimate the occupation time of deep valleys previously introduced, 
it is relevant to condition the potential to be nonnegative on $\Z_-.$ The main result of this section is the following.

\begin{proposition} 
 \label{prop:occupationdeep}
 The tail distribution of the hitting time of the first negative record $e_1$ satisfies
\begin{equation*}
 t^{\kappa} \,  \bPp \left(\tau(e_1) \ge t \right)  \longrightarrow C_T, \qquad t \to \infty,
  \end{equation*}
  where the constant $C_T$ is given by
  \begin{equation}\label{eqn:c_t=}
C_T\defeq 2^\kappa \, \Gamma(\kappa+1) C_U.
  \end{equation}
  \end{proposition}
The idea of the proof is the following. We show first that the the height of the first excursion has to be larger than a function $h_t$  (of order $\log t$). Secondly, we prove that conditional on  $H \ge h_t$ the environment has locally ``good'' properties. Finally, we decompose the passage from $0$ to
$e_1$ into the sum of a random geometrically distributed number of unsuccessful attempts to reach $e_1$ from $0$ (i.e. excursions of the particle  from $0$ to $0$ which do not hit $e_1$), followed by a successful attempt. This enables us to prove that  $\tau(e_1)$ behaves as an exponentially distributed random variable with mean $2Z$ where $Z$ is defined by $Z\defeq M_1 M_2 \, \ee^H$  and whose tail distribution is studied in \cite{renewal} and recalled in Lemma \ref{lem:tail_z}. 

In this proof, we denote $\tau(e_1)$ by $\tau$. 

\subsection{The height of the first excursion has to be large}

Let the critical height $h_t$ be a function of $t$ defined by
\begin{equation}
\label{h_T:def} 
h_t\defeq \log t - \log \log t, \qquad t \ge \ee^{\ee}.
\end{equation}

\begin{lemma} \label{lem:Hlarge} We have
\begin{equation*} \label{eq:estimatevalley1}
\bPp (\tau(e_1)>t \,; \, H \le h_t)=o(t^{-\kappa}), \qquad t \to \infty.
\end{equation*}
\end{lemma}

\begin{proof}
Let us first assume that $0<\kappa<1$. Then, by Markov inequality, we get
\begin{align*}
\bPp(\tau>t, H\leq h_t)
	& = \Ep[P_\omega(\tau>t)\indic{H\leq h_t}] \leq \frac{1}{t} \Ep[E_\omega[\tau]\indic{H\leq h_t}]\\
	& \leq \frac{1}{t} \Ep[2M_1'M_2\ee^H \indic{H\leq h_t}] \leq \frac{1}{t} C \ee^{(1-\kappa)h_t},
\end{align*}
where the last inequality follows from Lemma~\ref{lem:renewalestimates}. Since $t^{-1}\ee^{(1-\kappa)h_t}=t^{-\kappa}(\log t)^{-(1-\kappa)}$, this settles this case. 

Let us now assume $1<\kappa<2$. By Markov inequality, we get
\begin{align*}
\bPp(\tau>t, H \le h_t)
	&  \le \frac{1}{t^2} \Ep[E_\omega[\tau^2]\indic{H \le h_t}].
\end{align*}
Applying Lemma \ref{lem:tau_e_1} yields $\bPp (\tau>t, H \le h_t) \le Ct^{-2} \ee^{(2-\kappa)h_t}$, which concludes the proof of Lemma  \ref{lem:Hlarge} when $\kappa\neq 1$.


For $\kappa=1$, neither of the above techniques works: the first one is too rough, and $\Var_\omega(\tau)$ is not integrable. We shall modify $\tau$ so as to make $\Var_\omega(\tau)$ integrable. To this end, let us refer to Subsection~\ref{subsec:reduc_small} and denote by $d_-$ the right end of the first excursion on the left of 0 that is higher than $h_t$, and by $\taut\defeq\tautx{d_-}(0,e_1)$ the time spent on the left of $d_-$ before reaching $e_1$. By Lemma~\ref{lem:remontee_difficile} we have $\bEp[\taut\indic{H<h_t}]\leq C h_t\ee^{-h_t}\leq C(\log t)^2t^{-1}$. Let us also introduce $\taut'$, which is defined like $\taut$ but in the modified environment, i.e.\ by replacing the high excursions (on the left of $d_-$) by small ones (cf.~after Lemma~\ref{lem:remontee_difficile}). Then we have
\begin{align*}
\bPp(\tau>t,H<h_t)
	& \leq \bPp(\taut>(\log t)^3,H<h_t)+\bPp(\tau-\taut>t-(\log t)^3, H<h_t)\\
	& \leq \frac{1}{(\log t)^3}\bEp[\taut\indic{H<h_t}] + \bPp(\tau-\taut+\taut'>t-(\log t)^3, H<h_t)\\
	& = o(t^{-1}) + (\bPp)'(\tau>t-(\log t)^3, H<h_t)\\
	& \leq  o(t^{-1}) + \frac{1}{(t-(\log t)^3)^2}(\Ep)'[E_\omega[\tau^2]\indic{H<h_t}],
\end{align*}
and Lemma \ref{lem:tau_e_1} allows us to conclude just like in the case $1<\kappa<2$. 

\begin{remark*}
An alternative proof for $\kappa=1$, avoiding the use of a modified environment, would consist in bounding the heights of all excursions on the left of 0 by increasing quantities so as to give this event overwhelming probability; this method is used after~\eqref{eqn:r^-}. 
\end{remark*}
\end{proof}

\subsection{``Good'' environments}
Let us introduce the following events  
 \begin{align*}
\Omega^{(1)}_t &\defeq   \left\{ e_1 \le C \log t \right\},\\
\Omega^{(2)}_t &\defeq    \left\{ \max\{ -V^{\downarrow}(0,T_H) \, ; \, V^{\uparrow}(T_H,e_1)\} \le \alpha \log t  \right\},\\
\Omega^{(3)}_t &\defeq   \left\{ R^- \le   (\log  t)^4 t^\alpha \right\},
\end{align*}
where $\max(0,1-\kappa)<\alpha<\min(1,2-\kappa)$ is arbitrary, and $R^-$ will be introduced in Subsection \ref{subsec:hproc}. 
Then, we define the set of ``good'' environments at time $t$ by 
\begin{equation*}\label{eq:goodenv}
 \Omega_t\defeq     \Omega^{(1)}_t \cap  \Omega^{(2)}_t \cap  \Omega^{(3)}_t.
  \end{equation*}
The following result tells that ``good'' environments are asymptotically typical.
 \begin{lemma}
 \label{lem:goodenv} 
The event $\Omega_t$ satisfies 
\begin{equation*}
\label{eq:estimatevalley2}
P( \Omega_t^c \,; \, H \ge h_t)=o(t^{-\kappa}), \qquad t \to \infty.
 \end{equation*}
\end{lemma}

The proof of this result is easy but technical and postponed to the Appendix.

\subsection{Preliminary results: two $h$-processes}
\label{subsec:hproc}

In order to estimate finely the time spent in a deep valley, we decompose the passage from $0$ to $e_1$ into the sum of a random geometrically distributed number, denoted by $N$, of unsuccessful attempts to reach $e_1$ from $0$ (i.e. excursions of the particle  from $0$ to $0$ which do not hit $e_1$), followed by a successful attempt.   More precisely, since $N$ is a geometrically distributed random variable with parameter $1-p$
satisfying
\begin{equation}
\label{1-p} 
1-p = \frac{\omega_0}{\sum_{x=0}^{e_1-1}
\ee^{V(x)}}= \frac{\omega_0}{M_2 \ee^H},
\end{equation}
 we can write $ \tau(e_1)=\sum_{i=1}^{N} F_i +G,$ where the $F_i$'s are the durations of the successive i.i.d. failures and $G$ that of the first success. The accurate estimation of the time spent by each (successful and unsuccessful) attempt leads us to consider two $h$-processes where the random walker evolves in two modified potentials, one corresponding to the conditioning on a failure (see the potential $\widehat{V}$) and the other to the conditioning on a success (see the potential $\bar{V}$). Note that this approach was first introduced in \cite{limitlaws} to estimate the quenched Laplace transform of the occupation time of a deep valley in the case $0<\kappa<1.$

\subsubsection{The failure case: the $h$-potential $\widehat{V}$}
\label{Vcheck}

Let us fix a realization of $\omega.$ To introduce the $h$-potential
$\widehat{V},$ we  define $h(x)\defeq 
P_{x,\omega} (\tau(0) < \tau(e_1)).$ For any $0<x<e_1,$
we introduce $\widehat{\omega}_x\defeq \omega_x \frac{h(x+1)}{h(x)}.$ Since $h$ is a harmonic function, we have $1-\widehat{\omega}_x=(1-\omega_x) \frac{h(x-1)}{h(x)}.$ Note that $h(x)$ satisfies, see \eqref{eqn:zeitouni_p},
 \begin{equation}
\label{def:h(x)}
h(x) = {\sum_{k=x}^{e_1-1} \ee^{V(k)}}{\bigg(\sum_{k=0}^{e_1-1} \ee^{V(k)}\bigg)^{-1}} ,\qquad 0 < x < e_1.
\end{equation}
Now, $\widehat{V}$ can be defined for $x \ge 0$  by
$
\widehat{V}(x) \defeq   \sum_{i=1}^x \log \frac{1-\widehat{\omega}_i}{\widehat{\omega}_i}.
$
We obtain for any $0\le x<y<e_1,$
\begin{equation}
\label{Vcheckeq0}
\widehat{V}(y)-\widehat{V}(x)=\left(V(y)-V(x)\right)+ \log \bigg(\frac{
h(x) \, h(x+1) }{ h(y) \, h(y+1)} \bigg).
\end{equation}

\noindent Since $h(x)$ is a decreasing function of $x$ (by definition), we get
 for any $0\le x < y \le e_1,$
\begin{equation}\label{Vcheckeq2}
\widehat{V}(y)-\widehat{V}(x) \ge V(y)-V(x).
\end{equation}
From \cite{limitlaws} (see Lemma 12), we recall the following explicit computations for the first and second moments of $F.$ For any environment $\omega,$ we have
\begin{equation*}\label{expF}
   E_{\omega} \left[ F \right]= 2 \, \omega_0 \bigg( \sum_{i=-\infty}^{-1}
    \ee^{-V(i)} + \sum_{i=0}^{e_1-1}
    \ee^{-\widehat{V}(i)}\bigg) \defqe  2 \, \omega_0 \,  \widehat M_1,
\end{equation*}
and
\begin{equation}\label{expF^2}
    E_{\omega} \left[ F^2 \right]= 4 \omega_0 \, R^+ + 4(1-
    \omega_0) \, R^-,
\end{equation}
\noindent where $R^+$ and $R^-$ are defined by
\begin{align*}
  R^+ &\defeq  \sum_{i=1}^{e_1-1} \bigg(1+2 \sum_{j=0}^{i-2} \ee^{\widehat{V}(j)-\widehat{V}(i-1)} \bigg) \bigg( \ee^{-\widehat{V}(i-1)} +2 \sum_{j=i+1}^{e_1-1} \ee^{-\widehat{V}(j-1)} \bigg),\\
  R^- &\defeq  \sum_{i=-\infty}^{-1} \bigg(1+2 \sum_{j=i+2}^{0} \ee^{V(j)-V(i+1)} \bigg)\bigg( \ee^{-V(i+1)} +2 \sum_{j=-\infty}^{i-1} \ee^{-V(j+1)}\bigg).
\end{align*}

Moreover, we can prove the following useful properties.
\begin{lemma} 
 \label{lemmafailurebound} 
For all $t\ge1,$ we have on  $\Omega_t$
\begin{align}
\label{eq:boundVarF}
\Var_\omega(F) &\le C (\log  t)^4 t^\alpha,\\
M_2 &\le C \log t,\label{eq:boundM2}\\
\vert \widehat M_1 -M_1 \vert &\le o(t^{-\delta})M_1,\label{eq:boundM1}
\end{align}
with $\delta \in(0,1-\alpha).$
\end{lemma}

The proof of this result is postponed to the appendix.

\medskip


\subsubsection{The success case: the $h$-potential $\bar{V}$}
\label{Vbar}

 In a similar way, we introduce the $h$-potential $\bar{V}$
by defining $g(x)\defeq 
P_{x,\omega} (\tau(e_1) < \tau(0))=1-h(x).$ For any $0<x<e_1,$
we introduce $\bar{\omega}_x\defeq \omega_x \frac{g(x+1)}{g(x)}.$ Since $g$ is a harmonic function, we have $1-\bar{\omega}_x=(1-\omega_x) \frac{g(x-1)}{g(x)}.$  
Note that $g(x)$ satisfies, see \eqref{eqn:zeitouni_p}, 
 \begin{equation}
\label{def:g(x)}
g(x) = {\sum_{k=0}^{x-1} \ee^{V(k)}}{\bigg(\sum_{k=0}^{e_1-1} \ee^{V(k)}\bigg)^{-1}} ,\qquad 0 < x < e_1.
\end{equation}
Then, $\bar{V}$ can be defined for $x \ge 0$  by
$$
\bar{V}(x) \defeq   \sum_{i=1}^x \log \frac{1-\bar{\omega}_i }{\bar{\omega}_i}. $$
Moreover, for any $0<x<y\le e_1,$ we have
\begin{equation}
\label{Vbareq0} \bar{V}(y)-\bar{V}(x)=\left(V(y)-V(x)\right)+ \log
\bigg(\frac{ g(x) \, g(x+1) }{ g(y) g(y+1)} \bigg).
\end{equation}

\noindent Since $g(x)$ is a increasing function of $x,$ we get
 for any $0 \le x < y \le e_1,$
\begin{equation}
\label{Vbareq2}
   \bar{V}(y)-\bar{V}(x) \le V(y)-V(x).
\end{equation}
Moreover, we have for any environment $\omega$ (see \eqref{eqn:zeitouni_e2}),
\begin{equation}\label{expS2}
   E_{\omega}[G] \le   2 \sum_{0 \le i \le j < n }  \ee^{\bar V(j)- \bar V(i)}.
\end{equation}
Using this expression, we can use the ``good'' properties of the environment to obtain the following bound.
\begin{lemma} 
 \label{lemmafsuccessbound} 
For all $t\ge1,$ we have on  $\Omega_t$
\begin{equation*}
E_{\omega}[G] \le C (\log  t)^4 t^\alpha.
\end{equation*}
\end{lemma}

The proof of this result is again postponed to the appendix.

\medskip

\subsection{Proof of Proposition \ref{prop:occupationdeep}}

Recalling Lemma \ref{lem:Hlarge} and Lemma \ref{lem:goodenv},  the proof of Proposition \ref{prop:occupationdeep} boils down to showing that
\begin{equation*}
 t^{\kappa} P(H \ge h_t ) \,  \bP^{\overline \Omega_t} \left(\tau \ge t  \right) \longrightarrow C_T, \qquad t \to \infty,
\end{equation*}
where
\begin{equation*}
\overline \Omega_t \defeq  \Omega_t  \cap \{H \ge h_t\}  \cap \{ \forall x \le 0,\ V(x)\le 0\}.
\end{equation*}
Using the notations introduced in the previous subsections we can first write
 \begin{equation}
  \bP^{\overline \Omega_t} \left(\tau \ge t  \right) = E^{\overline \Omega_t} \left[  P_\omega \left(\tau \ge t \right)  \right]  =  E^{\overline \Omega_t} \left[ \sum_{k\ge0} (1-p) p^k P_{\omega} (\sum_{i=1}^k F_i +G \ge t )  \right].
  \label{eq:decompoechec}
\end{equation}
Moreover, note that we will use $\epsilon_t$ in this subsection to denote a function which tends to $0$ when $t$ tends to infinity but whose value can change from line to line.

\subsubsection{Proof of the lower bound} Let us introduce $
\xi_t\defeq (\log t)^{-1}$ for  $t \ge \ee$ and
 \begin{equation}
 \label{def:K+}
 K_+\defeq   \frac{ \ee^{\xi_t} t }{E_\omega \left[F\right]}=  \frac{ \ee^{\xi_t} t}{2 \omega_0 \widehat M_1}.
\end{equation}
Since the random variable $G$ is nonnegative,  the sum in \eqref{eq:decompoechec} is larger than
 \begin{align}
  & \sum_{k\ge0} (1-p) p^k P_{\omega} (\sum_{i=1}^k F_i \ge t \, ; \, k \ge K_+)\notag\\
  &\ge p^{K_+} -(1-p) \sum_{k\ge K_+} p^k P_{\omega} (\sum_{i=1}^k F_i \le t  \, ; \, k \ge K_+).
  \label{eq:probaterm1}
\end{align}
Now  for any $k \ge K_+,$ the probability term in \eqref{eq:probaterm1} is less than
\begin{align}
 & P_{\omega} (\sum_{i=1}^k F_i \le \ee^{-\xi_t} k E_\omega \left[F\right]) \le \frac{\Var_\omega(F)}{k E_\omega \left[F\right]^2 (1-\ee^{-\xi_t})^2 } \notag\\
 &\le \frac{\Var_\omega(F)}{K_+ E_\omega \left[F\right]^2 (1-\ee^{-\xi_t})^2 }
 \le  \frac{\Var_\omega(F)}{ t (1-\ee^{-\xi_t})^2 },
   \label{eq:probaterm2}
  \end{align}
the last inequality being a consequence of the definition of $K_+$ given by \eqref{def:K+} together with the fact that $ \ee^{\xi_t}  E_\omega \left[F\right] \ge 1.$ Therefore, assembling \eqref{eq:probaterm1} and \eqref{eq:probaterm2} yields
 \begin{equation*}
  \bP^{\overline \Omega_t} \left(\tau \ge t  \right) \ge E^{\overline \Omega_t} \left[  \left(1-  \frac{\Var_\omega(F)}{ t (1-\ee^{-\xi_t})^2 } \right)  p^{K_+}\right].
\end{equation*}
Since $\alpha<2-\kappa<1,$ Lemma  \ref{lemmafailurebound} implies
 \begin{equation}
  \label{eq:mino1}
  \bP^{\overline \Omega_t} \left(\tau \ge t  \right) \ge  \left(1-\epsilon_t \right)  E^{\overline \Omega_t} \left[  p^{K_+}\right].
\end{equation}
Furthermore, recalling \eqref{1-p} and  \eqref{def:K+} yields
 \begin{equation*}
E^{\overline \Omega_t} \left[  p^{K_+}\right]= E^{\overline \Omega_t} \left[  (1- \frac{\omega_0}{M_2 \ee^H})^{ \frac{ \ee^{\xi_t} t }{2 \omega_0 \widehat  M_1}} \right]
\ge  E^{\overline\Omega_t} \left[  (1- \frac{\omega_0}{M_2 \ee^H})^{ \frac{ \ee^{\xi'_t} t }{2 \omega_0 M_1}} \right],
 \end{equation*}
where the inequality is a consequence of Lemma \ref{lemmafailurebound}  and $\xi'_t\defeq \xi_t- \log (1-o(t^{-\delta}))= \xi_t +o(t^{-\delta}).$  Then, observe that $\omega_0 / M_2 \ee^H \le \ee^{-h_t} $ and recall that $\log (1-x) \ge -x(1+x),$ for $x$ small enough, such that we obtain
 \begin{equation*}
E^{\overline \Omega_t} \left[  p^{K_+}\right]\ge E^{\overline \Omega_t}  \left[ \exp\left\{- \frac{ \ee^{\xi'_t} t }{2 Z} (1+ \frac{\omega_0}{M_2 \ee^H})  \right\} \right]  \ge  E^{\overline \Omega_t} \left[ \ee^{- \frac{ \ee^{\xi''_t} t }{2 Z}  } \right],
\end{equation*}
where we recall that $Z= M_1 M_2 \, \ee^{H}$ and  $\xi''_t\defeq \xi'_t + \log (1+\ee^{-h_t} )= \xi_t+o(t^{-\delta}).$ Moreover Lemma \ref{lem:goodenv}  implies
 \begin{equation*}
  E^{\overline \Omega_t} \left[ \ee^{- \frac{ \ee^{\xi''_t} t}{2 Z}  } \right] 
\ge (1-\epsilon_t)  E^{ \Omega_t^*} \left[ \ee^{- \frac{ \ee^{\xi''_t} t}{2 Z}  } \right] - \epsilon_t \frac{t^{-\kappa}}{P(H \ge h_t)},
 \end{equation*} 
 where
\begin{equation*}
 \Omega_t^* \defeq  \{H \ge h_t\}  \cap \{ V(x) \ge 0, \, \forall x \le 0\}.
\end{equation*}
Now, we would like to integrate with respect to $Z.$ To this goal, let us introduce the notation $F_Z^{(t)}(z)\defeq  P^{\ge0}(Z>z \, | \, H \ge h_t ).$ An integration by part yields
 \begin{align}
 \nonumber
E^{\Omega_t^*}\left[ \ee^{- \frac{ \ee^{\xi''_t} t}{2 Z}  } \right] 
	& = \int_{\ee^{h_t}}^{\infty}  \ee^{- \frac{ \ee^{\xi''_t} t }{2 z}  }     \d F_Z^{(t)}(z)\\
	&= -  \ee^{- \frac{ \ee^{\xi''_t} t }{2 \ee^{h_t}}  }  F_Z^{(t)}(\ee^{h_t} ) +  \int_{\ee^{h_t}}^{\infty}  \frac{ \ee^{\xi''_t} t  }{2 z^2}   \ee^{- \frac{ \ee^{\xi''_t} t}{2 z}  }  F_Z^{(t)}(z) \d z.  \label{Ipp}
 \end{align}
Then, let us make the crucial observation that
 \begin{equation*}
 F_Z^{(t)}(z)= \frac{P^{\ge0}(Z>z )}{P( H \ge h_t )} - \frac{P^{\ge0}(Z>z \, ; \, H < h_t )}{P(H \ge h_t )}.
 \end{equation*} 
 Therefore, denoting by $I$ the integral in \eqref{Ipp}, we can write $I=I_1-I_2$, where $I_1$ and $I_2$ are given by
  \begin{align*}
 I_1 &\defeq \frac{1}{P( H \ge h_t )} \int_{\ee^{h_t}}^{\infty}  \frac{ \ee^{\xi''_t} t }{2 z^2} \ee^{- \frac{ \ee^{\xi''_t} t}{2 z}  }  P^{\ge0}(Z>z ) \d z,\\
  I_2 &\defeq \frac{1}{P(H \ge h_t )}  \int_{\ee^{h_t}}^{\infty}  \frac{ \ee^{\xi''_t} t}{2 z^2} \ee^{- \frac{ \ee^{\xi''_t} t}{2 z}  } P^{\ge0}(Z>z \, ; \, H < h_t ) \d z.
 \end{align*}
 To treat $I_1,$ let us recall that Lemma~\ref{lem:tail_z} gives the tail behaviour of $Z$ under $\Pp$:
   \begin{equation}
   \label{eq:tailZ}
 (1- \epsilon_t) C_U z^{-\kappa}   \le   P^{\ge0}(Z>z )  \le (1+ \epsilon_t) C_U z^{-\kappa},
 \end{equation}
for all $z\ge \ee^{h_t}.$  Hence, we are led to compute the integral
  \begin{equation}
 \int_{\ee^{h_t}}^{\infty}  \frac{ \ee^{\xi''_t} t}{2 z^2}   \ee^{- \frac{ \ee^{\xi''_t} t  }{2 z}  } z^{-\kappa} \d z =\ee^{- \kappa \xi''_t} {2^\kappa} \left(  \int_{0}^{ \frac{ \ee^{\xi''_t} t  }{2}}    \ee^{- y  } y^{\kappa} \d y \right)    t^{-\kappa},
 \label{eq:intgamma}
 \end{equation}
by making the change of variables given by $y=\ee^{\xi''_t} t /2 z.$ Observe that the integral in \eqref{eq:intgamma} is close to $\Gamma(\kappa+1)$ when $t$ tends to infinity (indeed $\xi''_t \to 0$). Therefore, recalling \eqref{iglehartthm}  and that $C_T=C_U 2^\kappa \Gamma(\kappa+1)$, we obtain
   \begin{equation}
     \label{eq:I1}
 (1- \epsilon_t) C_T t^{-\kappa}   \le    I_1 \, P(H \ge h_t )  \le (1+ \epsilon_t) C_T t^{-\kappa}.
 \end{equation}
 
We turn now to $I_2.$ Repeating the proof of Corollary $4.2$ in \cite{renewal} yields   
\begin{equation*}
P^{\ge0}(Z>z \, ; \, H < h_t ) \le C z^{-\eta} \ee^{(\eta-\kappa) h_t},
 \end{equation*}
for any $\eta>\kappa$ and all $z \ge \ee^{h_t}.$
 Therefore, repeating the previous computation, we get
  \begin{equation}
  \label{eq:I2}
 I_2 \, P(H \ge h_t ) \le C 2^\eta  \left(  \int_{0}^{ \frac{ \ee^{\xi''_t} t  }{2}}    \ee^{- y  } y^{\eta} \d y \right)  \ee^{(\eta-\kappa) h_t}  t^{-\eta} 
\le C 2^\eta \Gamma(\eta+1) \ee^{(\eta-\kappa) h_t}  t^{-\eta},
 \end{equation}
which yields $ I_2 \, P(H \ge h_t ) \le  \epsilon_t t^{-\kappa},$ by choosing $\eta$ larger than $\kappa$ and recalling \eqref{h_T:def}.
 
  Then assembling \eqref{eq:I1} and \eqref{eq:I2} implies $ I  \, P(H \ge h_t )  \ge (1- \epsilon_t) C_T t^{-\kappa}$ and coming back to \eqref{eq:mino1}--\eqref{Ipp}, we obtain
    \begin{equation*}
  \label{eq:I2b}
P(H \ge h_t ) \,  \bP^{\overline \Omega_t}\left(\tau \ge t \right)   \ge  -  \ee^{- \frac{ \ee^{\xi''_t} t }{2 \ee^{h_t}}  } P(H \ge h_t ) + (1- \epsilon_t) C_T t^{-\kappa},
 \end{equation*}
which concludes the proof of the lower bound since  $ \exp\{- \frac{ \ee^{\xi''_t} t }{2\ee^{h_t}}  \} P(H \ge h_t )  =o(t^{-\kappa})$ 
when $t$ tends to infinity; indeed $y^{\kappa} \ee^{-c y} \to 0$ when $y \to \infty$ and $t^{-1} \ee^{h_t} \to 0$ when $t \to \infty,$ see \eqref{h_T:def}.

\subsubsection{Proof of the upper bound}

Using still the notations $
\xi_t\defeq (\log t)^{-1}$ for  $t \ge \ee$, let us now introduce
 \begin{equation*}
 K_-\defeq   \frac{ \ee^{-\xi_t} t}{E_\omega \left[F\right]}=  \frac{ \ee^{-\xi_t} t}{2 \omega_0 \widehat M_1}.
 \label{eq:defK-}
\end{equation*}
Let also $\eta_t\defeq\xi_t-\frac{1}{2}\xi_t^2$, so that $0<\eta_t<1-\ee^{-\xi_t}$. The sum in \eqref{eq:decompoechec} is smaller than
 \begin{align}
 & p^{K_-} +(1-p) \sum_{k\le K_-} p^k P_{\omega} (\sum_{i=1}^k F_i + G \ge t) \notag\\
  &\le p^{K_-} + \frac{E_{\omega}[G] }{\eta_t  t} + (1-p) \sum_{k\le K_-} p^k P_{\omega} (\sum_{i=1}^k F_i \ge t (1-\eta_t) ),
  \label{eq:probaterm11}
\end{align}
the inequality being a consequence of Chebychev inequality.
Furthermore, observe that  $k \le K_-$ implies $t\geq ke^{\xi_t}E_\omega[F]$ hence the probability term in \eqref{eq:probaterm11} is less than
\begin{equation*}
P_\omega(\sum_{i=1}^k F_i -kE_\omega[F]\geq k(\ee^{\xi_t}(1-\eta_t)-1)E_\omega[F]) \leq \frac{\Var_\omega(F)}{k(\ee^{\xi_t}(1-\eta_t)-1)^2 E_\omega[F]^2}
  \end{equation*}
(remembering $1-\eta_t>\ee^{-\xi t}$). Therefore, 
\begin{equation*}
P_\omega(\tau\geq t)
	 \leq p^{K_-} + \frac{E_\omega[G]}{\eta_t t}  + \frac{\Var_\omega(F)}{(\ee^{\xi_t}(1-\eta_t)-1)^2 E_\omega[F]^2}\sum_{k\leq K_-} \frac{(1-p)p^k}{k}.
\end{equation*}
The last sum is less than $(1-p)\log\frac{1}{1-p} = \frac{\omega_0}{M_2 \ee^H}\log\frac{M_2 \ee^H}{\omega_0}$. On the event $\overline{\Omega}_t$, we have $\ee^H\geq\ee^{h_t}$, $E_\omega[F]\geq 1$, $M_2\geq 1$, $\frac{1}{2}\leq \omega_0\leq 1$, and \eqref{eq:boundVarF}, hence 
\begin{equation*}
\bP^{ \overline \Omega_t} \left(\tau \ge t\right) \le E^{\overline  \Omega_t} \left[p^{K_-} \right] + E^{\overline  \Omega_t} \left[ \frac{E_{\omega}[G] }{\eta_t t} \right] + \frac{C(\log t)^4 t^\alpha}{(\ee^{\xi_t}(1-\eta_t)-1)^2}\frac{1}{\ee^{h_t}} E^{\overline  \Omega_t} \left[ \log(2M_2\ee^H) \right].
\end{equation*}
Let us now bound the three terms in the right-hand side of the previous equation. Consider the last one. Using Lemma \ref{lem:renewalestimates} and~\eqref{iglehartthm}, we have $E[\log(M_2 \ee^H)|H\geq h_t]=E[\log M_2|H\geq h_t]+E[H|H\geq h_t]\leq C+h_t$ for some constant $C$. When $t\to\infty$, $\ee^{\xi_t}(1-\eta_t)-1\sim \frac{\xi_t^3}{6}=\frac{1}{6\log^3 t}$. Since $\ee^{h_t}=\frac{t}{\log t}$ and $\alpha< 1$, the whole term is seen to converge polynomially to zero. In particular, 
\begin{equation} \label{eq:up3}
\frac{C(\log t)^4 t^\alpha}{(\ee^{\xi_t}(1-\eta_t)-1)^2}\frac{1}{\ee^{h_t}} E^{\overline  \Omega_t} \left[ \log(2M_2\ee^H) \right]\leq \epsilon_t  \frac{t^{-\kappa}}{P(H\geq h_t)}. 
\end{equation}

For the second term, Lemma \ref{lemmafsuccessbound} implies
 \begin{equation}
 \label{eq:up2}
E^{\overline \Omega_t} \left[ \frac{E_{\omega}[G] }{\xi_t t }  \right] \le C \frac{(\log t)^4 t^\alpha}{\xi_t t }   \le \epsilon_t \frac{t^{-\kappa}}{P(H\ge h_t)},
  \end{equation}
  since $\alpha<1$. 
Finally, for the first expectation, we repeat the arguments of the proof of the upper bound obtained for $I.$ More precisely, recalling \eqref{1-p} and  \eqref{def:K+}, we get
 \begin{equation*}
E^{\overline \Omega_t} \left[  p^{K_-}\right] \le  E^{\overline \Omega_t} \left[  (1- \frac{\omega_0}{M_2 \ee^H})^{ \frac{ \ee^{-\xi'_t} t}{2 \omega_0 M_1}} \right] \le  E^{\overline \Omega_t} \left[ \ee^{- \frac{ \ee^{-\xi'_t} t}{2 Z}  } \right],
 \end{equation*}
where the first inequality is a consequence of Lemma \ref{lemmafailurebound} and $\xi'_t\defeq \xi_t- \log (1+o(t^{-\delta}))= \xi_t +o(t^{-\delta}),$ while the second inequality is a consequence of  $\log (1-x) \le -x$ for $0<x<1.$  Then, an integration by part yields
   \begin{equation*}
E^{\overline \Omega_t} \left[  p^{K_-}\right]    \le  \frac{1+ \epsilon_t}{P( H \ge h_t)} \int_{\ee^{h_t}}^{\infty}  \frac{ \ee^{-\xi'_t} t}{2 z^2}   \ee^{- \frac{ \ee^{-\xi'_t} t }{2 z}  }  P^{\ge0}(Z>z ) \d z.
 \end{equation*}
Making the change of variables given by $y=\ee^{-\xi'_t} t  /2 z$ and recalling \eqref{eq:tailZ} imply
   \begin{equation}
    \label{eq:up1}
P( H \ge h_t) \, E^{\overline \Omega_t} \left[  p^{K_-}\right]    \le  (1+ \epsilon_t) C_T t^{-\kappa}.
 \end{equation}
Now, assembling \eqref{eq:up3}, \eqref{eq:up2} and \eqref{eq:up1} concludes the proof of the upper bound.

\section{Proof of Theorem \ref{T:MAIN}}\label{sec:preuve_thm}

The results from Sections~\ref{sec:iid_valleys} and \ref{sec:interarrival} enable us to reduce the proof of Theorem \ref{T:MAIN} to an equivalent i.i.d.~setting and thus to apply a classic limit theorem. 

NB: we first prove the theorem under $\bPp$, and the statement under $\bP$ will follow.

\subsection{Reduction to i.i.d.~random variables}
For all $i\geq 0$, let $Z_i\defeq\tau(e_i,e_{i+1})$, so that $(Z_i)_{i\geq 0}$ is a stationary sequence under $\bPp$ (cf.~Lemma \ref{lem:e_negatifs}) and
\begin{equation*}
\tau(e_n)=Z_0+\cdots+Z_{n-1}.
\end{equation*}
Let us also enlarge the probability space $(\Omega\times\Z^\N,\mathcal{B},\bPp)$ in order to introduce an i.i.d.~sequence $(\omega^{(i)},(X^{(i)}_n)_{n\geq 0})_{i\geq 0}$ of environments and random walks distributed according to $\bPp$. Since the excursions of $V$ are independent, it is possible to couple $\omega$ and $(\omega^{(i)})_{i\geq 0}$ in such a way that, for all $i\geq0$, $H^{(i)}=H_i$, or more generally that the first excursion of $\omega^{(i)}$ and the $(i+1)$-th excursion of $\omega$ are the same. It suffices indeed to build $\omega^{(i)}$ from the excursion $(\omega_{e_i+x})_{1\leq x\leq e_{i+1}-e_i}$ of $\omega$
 and from independent environments on both sides of it. 


For all integers $i\geq 0$, we may now introduce
\begin{equation*}
\widehat{Z}_i\defeq\tau^{(i)}(e_1^{(i)})
\end{equation*}
which is defined like $Z_1(=\tau(e_1))$ but relatively to $(\omega^{(i)},X^{(i)})$ instead of $(\omega,X)$. By construction, $(\widehat{Z}_i)_{i\geq 0}$ is a sequence of i.i.d.~random variables distributed like $Z_1$ under $\bPp$. 

\paragraph*{For $1<\kappa<2$} We have the decomposition (where indices $i$ range from $0$ to $n-1$)
\begin{align}
\tau(e_n)-\bEp[\tau(e_n)]
	& =  \bigg(\sum_{H_i< h_n}Z_i - \bEp\bigg[\sum_{H_i< h_n}Z_i\bigg]\bigg) \notag\\
	&  + \bigg(\sum_{H_i\geq h_n}Z_i\bigg)\ind_{\overlap} + \bigg(\sum_{H_i\geq h_n}\taut_i\bigg)\ind_{\nonoverlap}\label{eqn:decomposition_tau_e_n}\\
	&  + \bigg(\sum_{H_i\geq h_n}Z_i^*\bigg)\ind_{\nonoverlap} -\bEp\bigg[\sum_{H_i\geq h_n}Z_i\bigg], \notag
\end{align}
where, if $H_i\geq h_n$ and $j$ is such that $\sigma(j)=i$ (i.e.~$e_i=b_j$), $\taut_i=\tautx{a_j}(e_i,e_{i+1})$ is the time spent on the left of $a_j$ after the first visit of $e_i$ and before reaching $e_{i+1}$, and $Z_i^*=Z_i-\taut_i$. 

Due to Propositions \ref{prop:tau_ia}, \ref{prop:overlap} and \ref{prop:lim_taut} respectively, the first three terms are negligible in $\bPp$-probability with respect to $n^{1/\kappa}$, hence 
\begin{equation*}
\frac{\tau(e_n)-\bEp[\tau(e_n)]}{n^{1/\kappa}}
	= \frac{1}{n^{1/\kappa}}\bigg(\bigg(\sum_{H_i\geq h_n}Z_i^*\bigg)\ind_{\nonoverlap} -\bEp\bigg[\sum_{H_i\geq h_n}Z_i\bigg]\bigg)+o(1), 
\end{equation*}
where $o(1)$ is a random variable converging to 0 in $\bPp$-probability. 

\paragraph*{For $\kappa=1$} Let
\begin{equation*}
a_n\defeq\inf\big\{t>0:\;\bPp(\tau(e_1)>t)\leq t^{-1}\big\}. 
\end{equation*}
(Note that $a_n\sim_n C_T n$ by Proposition \ref{prop:occupationdeep}). With the same definitions as above, we decompose
\begin{align*}
\tau(e_n)-n\bEp[\tau(e_1)\indic{\tau(e_1)<a_n}]
	& =  \bigg(\sum_{H_i< h_n}Z_i - \bEp\bigg[\sum_{H_i< h_n}Z_i\bigg]\bigg) \notag\\
	& \hspace{-2cm} + \bigg(\sum_{H_i\geq h_n}Z_i\bigg)\ind_{\overlap} + \bigg(\sum_{H_i\geq h_n}\taut_i\bigg)\ind_{\nonoverlap}\\
	& \hspace{-2cm} + \bigg(\sum_{H_i\geq h_n}Z_i^*\bigg)\ind_{\nonoverlap} -n\bEp[Z_1(\indic{Z_1<a_n,H\geq h_n}-\indic{Z_1\geq a_n,H<h_n})]. \notag
\end{align*}
Note that the last term accounts for the difference between the restriction according to the value of $\tau(e_1)$, used on the left-hand side and that we need for applying the limit theorem, and the restriction according to the height, used in the right-hand side decomposition and throughout the paper. 

Again, the first three terms are negligible with respect to $n$, hence $n^{-1}(\tau(e_n)-n\bEp[\tau(e_1)\indic{\tau(e_1)<a_n}])$ equals
\begin{align}
	\frac{1}{n}\bigg(\bigg(\sum_{H_i\geq h_n}Z_i^*\bigg)\ind_{\nonoverlap} -n\bEp[Z_1(\indic{Z_1<a_n,H\geq h_n}-\indic{Z_1\geq a_n,H<h_n})]\bigg)+o(1). \label{eqn:46468}
\end{align}

Let us resume to the general case $1\leq \kappa<2$. Observe that $Z_{\sigma(j)}^*$ is the time to go from $b_j$ to $d_j$ for a random walk reflected at $a_j$, hence it depends only on the environment between $a_j+1$ and $d_j$. On the other hand, under $P(\cdot|K_n=m,\nonoverlap)$, the pieces $(\omega_{b_j+x})_{a_j< b_j+x\leq d_j}$ of the environment, for $j=1,\ldots,m$, are i.i.d.~with same distribution as $(\omega_x)_{e_{-D_n}<x\leq e_1}$ under \[\Pp(\cdot|H\geq h_n, H_{-k}<h_n \text{ for } k=1,\ldots,D_n).\] Remember indeed that $a_1>0$ on $\nonoverlap$; and due to our definition of deep valleys, conditioning by the value of $K_n$ only affects the number of deep valleys and not their individual distributions, while conditioning by $\nonoverlap$ implies the independence and imposes the excursions between $a_j$ and $b_j$ to be small, for $j=1,\ldots,K_n$. 

As a consequence, the term $\left(\sum_{H_i\geq h_n}Z_i^*\right)\ind_{\nonoverlap}$ has same distribution under $\bPp$ as $\left(\sum_{H_i\geq h_n}\widehat{Z}_i^*\right)\ind_{\nonoverlap}$ under $\bP(\cdot|\nonoverlapc)$, where $\widehat{Z}_i^*\indic{H_i\geq h_n}$ is defined like $Z^*_1\indic{H\geq h_n}$ but relative to $(\omega^{(i)},X^{(i)})$, and 
\begin{equation*}
\nonoverlapc\defeq\left\{\text{for }j=1,\ldots,K_n,\ H^{(\sigma(j))}_{-1}< h_n,\ldots,H^{(\sigma(j))}_{-D_n}< h_n\right\}
\end{equation*}
is the event that $D_n$ small excursions precede the high excursions in the i.i.d.~framework. 

We deduce that, for $1<\kappa<2$, the characteristic function satisfies 
\begin{align} \label{eqn:reduction1}
	& \hseq \bEp\Big[\ee^{i\lambda n^{-1/\kappa}(\tau(e_n)-\bEp[\tau(e_n)])}\Big]\notag\\
	& = \bEp\Bigg[\exp\bigg(i\lambda n^{-1/\kappa}(\bigg(\sum_{H_i\geq h_n}\widehat{Z}_i^*\bigg)\ind_{\nonoverlap}-\bEp\bigg[\sum_{H_i \geq h_n} Z_i\bigg])\bigg)\Bigg|\nonoverlapc\Bigg]+o_n(1)\notag\\
	& = \bEp\Bigg[\exp\bigg(i\lambda n^{-1/\kappa}(\left(\sum_{H_i \geq h_n}\widehat{Z}_i^*\right)\ind_{\nonoverlap}-\bEp\bigg[\sum_{H_i\geq h_n} Z_i\bigg])\bigg)\Bigg]+o_n(1)'. 
\end{align}
The last equality comes from $P(\nonoverlapc)\to_n 1$, cf.~Lemma \ref{lem:z_hat} below, and from the fact that the term in the expectation is bounded by 1. We have of course similar equalities for $\kappa=1$ from \eqref{eqn:46468}. 

The following lemma will enable us to put the neglected terms back in the sum, now with $\widehat{Z}_i$ instead of $Z_i$, and thus complete the reduction to i.i.d.~random variables. For $i\geq 0$, let $\widehat{\widetilde{\tau}}_i$ be the time spent by $X^{(i)}$ on the left of $e_{-D_n}$ ($=a_1$ if $H>h_n$) before $e_1$ is reached, hence $\widehat{Z}_i=\widehat{Z}^*_i+\widehat{\widetilde{\tau}}_i$. 

\begin{lemma}\label{lem:z_hat}
We have 
\begin{equation*}
P(\nonoverlapc)\limites{}{n}1,\label{eqn:overlap_hat}
\end{equation*}
\begin{equation*}
\frac{1}{n^{1/\kappa}}\sum_{i=0}^{n-1} \widehat{\widetilde{\tau}}_i \indic{H^{(i)}\geq h_n} \limites{(p)}{n} 0,\label{eqn:lim_taut_hat}
\end{equation*}
\begin{equation}
\frac{1}{n^{1/\kappa}}\left(\sum_{i=0}^{n-1} \widehat{Z}_i \indic{H^{(i)}<h_n}-E\bigg[\sum_{i=0}^{n-1} \widehat{Z}_i \indic{H^{(i)}<h_n}\bigg]\right) \limites{(p)}{n}0.\label{eqn:tau_ia_hat}
\end{equation}
\end{lemma}

\begin{proof} 
These results follow respectively from the proofs of Propositions \ref{prop:overlap}, \ref{prop:lim_taut} and \ref{prop:tau_ia}, made easier by the independence of the random variables $\widehat{Z}_0,\ldots,\widehat{Z}_{n-1}$. More precisely, the proofs of Propositions \ref{prop:overlap} and \ref{prop:lim_taut} hold in this i.i.d.~context almost without a change. And since the random variables $\widehat{Z}_i \indic{H^{(i)}<h_n}$, $i\geq 0$, are independent, the proof of \eqref{eqn:tau_ia_hat} for $1<\kappa<2$ would follow from
\begin{equation*}
n\bVarp(\tau(e_1)\indic{H<h_n})=o(n^{2/\kappa}), 
\end{equation*}
and thus from $n\bEp[\tau(e_1)^2\indic{H<h_n}]=o(n^{2/\kappa})$, which is given by Lemma~\ref{lem:tau_e_1}. For $\kappa=1$, the same modification of the environment as in Subsection \ref{subsec:reduc_small} adapts immediately. 
\end{proof}

From this lemma and \eqref{eqn:reduction1}, recomposing \eqref{eqn:decomposition_tau_e_n} with variables $\widehat{Z}_i$ (and using $\nonoverlap$ again, not $\nonoverlapc$), we finally have, for $1<\kappa<2$, 
\begin{equation} \label{eqn:reduc_iid}
\bEp\left[\ee^{i\lambda n^{-1/\kappa}(\tau(e_n)-\bEp[\tau(e_n)])}\right]=\bE\left[\ee^{i\lambda n^{-1/\kappa}(\widehat{Z}_0+\cdots+\widehat{Z}_{n-1}-\bE[\widehat{Z}_0+\cdots+\widehat{Z}_{n-1}])}\right]+o_n(1).
\end{equation}
Note that we used the equality $\bEp[\sum_{H_i>h_n} Z_i]=\bE[\sum_{H^{(i)}>h_n}\widehat{Z}_i]$, which results from the equality in distribution of $Z_i\indic{H_i\geq h_n}$ and $\widehat{Z}_i\indic{H^{(i)}\geq h_n}$ under $\bPp$. 

As a conclusion, this shows that, for $1<\kappa<2$, $\frac{\tau(e_n)-\bEp[\tau(e_n)]}{n^{1/\kappa}}$ has same limit in law under $\bPp$ (if any) as $\frac{\widehat{Z}_0+\cdots+\widehat{Z}_{n-1}-n\bEp[\widehat{Z}_0]}{n^{1/\kappa}}$, where the random variables $\widehat{Z}_i$, $i\geq 0$, are i.i.d.~with same distribution as $\tau(e_1)$ under $\bPp$. 

For $\kappa=1$, the same procedure shows that $\frac{\tau(e_n)-n\bEp[\tau(e_1)\indic{\tau(e_1)<a_n}]}{n}$ has same limit in law under $\bPp$, if any, as $\frac{\widehat{Z}_0+\cdots+\widehat{Z}_{n-1}-n\bEp[\widehat{Z}_0\indic{\widehat{Z}_0<a_n}]}{n}$.

\subsection{Conclusion of the proof}

Let us quote (a particular case of) Theorem 2.7.7 from \cite{durrett}:

\begin{theorem}\label{thm:durrett}
Suppose $X_1, X_2,\ldots$ are i.i.d.~nonnegative random variables with a distribution that satisfies
\begin{equation*}
\bP(X_1>x)=x^{-\alpha}L(x)
\end{equation*}
where $1\leq\alpha<2$ and $L$ is slowly varying. Let $S_n=X_1+\cdots+X_n$, 
\begin{equation*}
a_n=\inf\big\{x:\bP(X_1>x)\leq n^{-1}\big\}\quad\mbox{ and }\quad b_n=n\bE[X_1\indic{X_1<a_n}].
\end{equation*}
Then, if $1<\alpha<2$, 
\begin{equation*}
\frac{S_n-nE[X_1]}{a_n}\limites{\rm (law)}{n} (-\Gamma(1-\alpha))^{1/\alpha}\mathcal{S}^{ca}_\alpha, 
\end{equation*}
where $\mathcal{S}^{ca}_\alpha$ is a centered completely asymmetric stable random variable of index $\alpha$, defined in \eqref{eq:defStable}. 

And if $\alpha=1$, 
\begin{equation*}
\frac{S_n-b_n}{a_n}\limites{\rm (law)}{n} c+\mathcal{S}^{ca}_1, 
\end{equation*}
where $c=1-\gamma$ ($\gamma\simeq 0.577$ being Euler's constant), and $\mathcal{S}^{ca}_1$ was defined in \eqref{eq:defStable1}. 
\end{theorem}

\begin{remarks}\leavevmode
\begin{itemize}
	\item Durrett~\cite{durrett} actually gives a different parametrization of the limit law. The above parameters are obtained by comparing the real and imaginary parts of expressions (7.11) and (7.13) (where there is a sign error) of \cite{durrett}, using the following identities: $\int_0^\infty\frac{1-\cos x}{x^{\alpha+1}}\d x = \cos\left(\frac{\pi\alpha}{2}\right)\Gamma(1-\alpha)$ (for any $0<\alpha<2$), and $\int_0^1\frac{\sin u-u}{u^2}\d u+\int_1^\infty\frac{\sin u}{u}\d u=1-\gamma$. The value of $c$ is however unimportant in the following. 
	\item If $\bP(X_1>x)\sim_{x\to\infty}\frac{C}{x^\alpha}$, then we have $a_n\sim_n C^{1/\alpha}n^{1/\alpha}$. 
\end{itemize}
\end{remarks}

Thanks to Proposition \ref{prop:occupationdeep} and to the previous reduction \eqref{eqn:reduc_iid} to an i.i.d.~framework, Theorem~\ref{thm:durrett} gives that, for $1<\kappa<2$, 
\begin{equation*}
\mbox{under }\bPp,\ \frac{\tau(e_n)-\bEp[\tau(e_n)]}{n^{1/\kappa}}\limites{\rm (law)}{n}(-\Gamma(1-\alpha)C_T)^{1/\kappa} \mathcal{S}^{ca}_\kappa.
\end{equation*}
The random walk is almost-surely transient to $+\infty$ under both $\bP$ and $\bPp$ (cf.~after \eqref{eqn:zeitouni_p3}), hence the total time spent on $\Z_-$ is finite in both cases and thus trivially negligible with respect to $n^{1/\kappa}$. Since random walks under distributions $\bP$ and $\bPp$ can simply be coupled so that they coincide after erasure of the time spent on $\Z_-$, we conclude that the above limit (with same centering) holds under $\bP$ as well. 

We deduce, using the law of large numbers and the central limit theorem for $(e_n)_n$ (cf.~the conclusion of \cite{kks}), 
\begin{equation}\label{eqn:4854}
\mbox{under }\bP,\ \frac{\tau(n)-nv^{-1}}{n^{1/\kappa}}\limites{\rm (law)}{n}(-\Gamma(1-\alpha)E[e_1]^{-1}C_T)^{1/\kappa} \mathcal{S}^{ca}_\kappa,
\end{equation}
where $v^{-1}\defeq\frac{1}{E[e_1]}\bEp[\tau(e_1)]$. Since \eqref{eqn:4854} yields $\frac{\tau(n)}{n}\to_n v^{-1}$ in probability, comparison with Solomon \cite{solomon} gives the value $v^{-1}=\bE[\tau(1)]=\frac{1+E[\rho_0]}{1-E[\rho_0]}$. By~\eqref{eqn:c_t=}, 
\[(-\Gamma(1-\kappa)E[e_1]^{-1}C_T)^{1/\kappa}=\left(-\Gamma(1-\kappa)2^\kappa\Gamma(1+\kappa)E[e_1]^{-1}C_U\right)^{1/\kappa},\]
and Euler's reflection formula  $\Gamma(1+\kappa)\Gamma(1-\kappa)=\frac{\pi\kappa}{\sin\pi\kappa}$, together with the expression of $C_U$ recalled in \eqref{eqn:C_U} leads to the value of Equation \eqref{eqn:thm_tau}. 

Finally, the limit law for $X_n$ results using transience to $+\infty$, cf.~\cite{kks}, pp.167--168. 

For $\kappa=1$, we get
\begin{equation*}
\mbox{under }\bPp,\ \frac{\tau(e_n)-n\bEp[\tau(e_1)\indic{\tau(e_1)<a_n}]}{n^{1/\kappa}}\limites{\rm (law)}{n} C_T(1-\gamma)+C_T\mathcal{S}^{ca}_1.
\end{equation*}
Furthermore, using Proposition \ref{prop:occupationdeep}, when $n\to\infty$, 
\begin{equation*}
\bEp[\tau(e_1)\indic{\tau(e_1)<a_n}]=\int_0^{a_n} \bPp(\tau(e_1)>t)\d t \sim C_T\log(a_n) \sim C_T\log n.
\end{equation*}
Like in the previous case, we may substitute $\bP$ for $\bPp$ (letting the centering term unchanged). Goldie \cite{goldie} proved that, when $\kappa=1$,  $C_K=\frac{1}{E[\rho_0\log\rho_0]}$, hence $C_T=\frac{2E[e_1]}{E[\rho_0\log\rho_0]}$. This concludes the proof of Theorem~\ref{T:MAIN} (cf.~\cite{kks} again for the inversion argument).

\section{Appendix}
\label{sec:appendix}

\subsection{Proof of Lemma \ref{lem:goodenv} }

Recalling the definition of $\Omega_t,$ the proof of Lemma \ref{lem:goodenv}  boils down to showing that for $i=1,2,3,$
\begin{equation}
\label{eq:envprincip}
P((\Omega_t^{(i)})^c \,; \, H \ge h_t)=o(t^{-\kappa}), \qquad t \to \infty.
\end{equation}

The case $i=1$ is trivial. Indeed, the fact that $e_1$ has some finite exponential moments (see after \eqref{eqn:def_e_i}) implies that 
$P((\Omega_t^{(1)})^c)=o(t^{-\kappa})$ when $t$ tends to infinity.

Furthermore, this result implies that the case $i=2$ is a consequence of  
\begin{equation*}
P((\Omega_t^{(2)})^c \,; \, \Omega_t^{(1)} \,; \, H \ge h_t)=o(t^{-\kappa}), \qquad t \to \infty.
\end{equation*}
Then, let us observe that $V^{\uparrow}(T_H,e_1)$ is less  than $V^{\uparrow}(T_{h_t},e_1)$ which is bounded by $V^{\uparrow}(T_{h_t},T_{h_t}+\lceil C \log t \rceil)$
 on $\Omega_t^{(1)}.$ Applying the strong Markov property at time $T_{h_t},$ we get that $P( V^{\uparrow}(T_H,e_1) \ge \alpha \log t  \,; \, \Omega_t^{(1)} \,; \, H \ge h_t)$ is bounded by 
 \begin{align*}
 \nonumber
P(H \ge h_t) \, P(V^{\uparrow}(0,\lceil C \log t \rceil) \ge \alpha \log t)
	&\le P(H \ge h_t) \, P\big(\max_{1\le k \le \lceil C \log t \rceil} H_k \ge \alpha \log t\big)\\
	&\le C (\log t) \, P(H \ge h_t) \, P( H \ge \alpha \log t).
\end{align*}
Recalling that $h_t=\log t - \log \log t,$ that $\alpha>0$ together with Iglehart's result yields
\begin{equation*}
P( V^{\uparrow}(T_H,e_1) \ge \alpha \log t  \,; \, \Omega_t^{(1)} \,; \, H \ge h_t)=o(t^{-\kappa}), \qquad t \to \infty.
\end{equation*}
Then to prove \eqref{eq:envprincip} for $i=2,$ it remains to show that 
\begin{equation}
\label{eq:fluctufin0}
P( V^{\downarrow}(0,T_H) \le - \alpha \log t \,; \, H \ge h_t)=o(t^{-\kappa}), \qquad t \to \infty.
\end{equation}
Observing that 
$V^{\downarrow}(0,T_H)=\min\{V^{\downarrow}(0,T_{h_t}) \, ; \, V^{\downarrow}(T_{h_t},T_H)\},$ we will treat each term separately.
From the trivial inclusion 
\begin{equation*}
\left\{
V^{\downarrow}(0,T_{h_t}) \le -\alpha \log t \, ; \, H \ge
h_t \right\} \!\subset\! \left\{ T^{\downarrow}(\alpha \log
t)\!<T_{h_t}\!< T_{(-\infty,0]} \right\},
\end{equation*}
 it follows that 
$P( V^{\downarrow}(0,T_{h_t}) \le - \alpha \log t   \,; \, H \ge h_t)$ is less than
\begin{equation*}
 \sum_{p=\lfloor
\alpha \log t \rfloor}^{ \lfloor h_t \rfloor} P(M_\alpha \in
[p,p+1) \, ; \, T^{\downarrow}(\alpha \log t)<T_{h_t}<
T_{(-\infty,0]}),
\end{equation*}
 where $M_\alpha\defeq \max\{V(k); \, 0 \le k \le
T^{\downarrow}(\alpha \log t)\}.$ Applying the strong Markov
property at time $T^{\downarrow}(\alpha \log t),$
 we bound the term of the previous sum
by $P( S \ge p ) \, P (S
\ge h_t-(p+1- \alpha \log t)).$ Then recalling that there exists $C$ such that $P( S \ge p ) \le C \ee^{-\kappa p}$ for all $p\ge 0$ (see \eqref{eqn:feller}),
we obtain the uniform bound $C \ee^{-\kappa
(h_t+\alpha \log t)}$ for  the summand, which yields
\begin{equation}
\label{eq:fluctufin1}
P( V^{\downarrow}(0,T_{h_t}) \le - \alpha \log t \,; \, H \ge h_t) \le C  h_t   \ee^{-\kappa
(h_t+\alpha \log t)}=o(t^{-\kappa}), \qquad t \to \infty,
\end{equation}
since $h_t= \log t - \log \log t$ and $\alpha>0.$
Furthermore, applying again the strong Markov property at
$T_{h_t},$ we obtain
\begin{equation*}
P( V^{\downarrow}(T_{h_t},T_H) \le - \alpha \log t\,; \, H \ge h_t) \le P( H \ge h_t )
 P(V^{\downarrow}(0,T_S) \le - \alpha \log t).
\end{equation*}
Then, applying the strong Markov
property at $T^{\downarrow}(\alpha \log t),$ we get that $P(V^{\downarrow}(0,T_S) \le - \alpha \log t)$ is less than $P(S > \alpha \log t),$ which yields
\begin{equation}
\label{eq:fluctufin2}
P( V^{\downarrow}(T_{h_t};T_H) \le - \alpha \log t \,; \, H \ge h_t) \le C  \ee^{-\kappa
(h_t+\alpha \log t)}  =o(t^{-\kappa}), \qquad t \to \infty.
\end{equation}
Now assembling \eqref{eq:fluctufin1} and \eqref{eq:fluctufin2} implies \eqref{eq:fluctufin0} and concludes the proof of the case $i=2.$

Let us consider the last case $i=3$. Since $R^-$ depends only on $\{V(x), \, x \le 0\}$, and $P(H>h_t)\sim C_I t^{-\kappa}(\log t)^{\kappa}$ when $t\to\infty$, it suffices to prove $\Pp(R^->(\log t)^4 t^\alpha)=o((\log t)^{-\kappa})$. This would follow (for any $\alpha>0$) from Markov property if $\Ep[R^-]<\infty$. We have (changing indices and incorporating the single terms into the sums): 
\begin{align}
R^-	&= \sum_{i\leq 0}\left(1+2\sum_{i\leq j\leq 0}\ee^{V(j)-V(i)}\right)\left(\ee^{-V(i)}+2\sum_{k\leq i-1}\ee^{-V(k)}\right)\notag\\
	&\leq 4\sum_{k\leq i\leq j\leq 0}\ee^{V(j)-V(i)-V(k)},\label{eqn:r^-} 
\end{align}
and this latter quantity was already seen to be integrable under $\Pp$, after \eqref{eqn:j_neg}, when $1<\kappa<2$. In order to deal with the case $0<\kappa\leq 1$, let us introduce the event
\begin{equation*}
A_t=\bigcap_{k=1}^\infty\{H_{-k}<\frac{1}{\kappa}\log k^2+\log t+\log\log t\}. 
\end{equation*}
On one hand, by \eqref{iglehartthm}, $P((A_t)^c)\leq \sum_{k=1}^\infty \frac{C}{k^2 (t\log t)^\kappa}=\left(\sum_{k=1}^\infty\frac{C}{k^2}\right)\frac{t^{-\kappa}}{(\log t)^\kappa}=o(t^{-\kappa})$. On the other hand, proceding like after \eqref{eqn:j_neg}, 
\begin{align*}
\Ep[R^-\ind_{A_t}]
	&\leq 4\sum_{u\leq 0}\Ep[\ee^{-V(e_u)}]\Ep[(M'_1)^2 M_2 \ee^H \indic{H<\frac{1}{\kappa}\log u^2+\log t+\log\log t}] 
\end{align*}
and $\Ep[\ee^{-V(e_u)}]=E[\ee^{V(e_1)}]^u$ hence, using Lemma \ref{lem:renewalestimates}, when $0<\kappa<1$, 
\begin{equation*}
\Ep[R^-\ind_{A_t}]
	\leq 4 \left(\sum_{u\leq 0}E[\ee^{V(e_1)}]^u \frac{1}{u^{2(1-\kappa)/\kappa}}\right) (t\log t)^{1-\kappa} = C(t\log t)^{1-\kappa},
\end{equation*}
and when $\kappa=1$, 
\begin{equation*}
\Ep[R^-\ind_{A_t}]
	\leq 4\sum_{u\leq 0}E[\ee^{V(e_1)}]^u(\frac{1}{\kappa}\log u^2+\log t+\log \log t)\leq C\log t.
\end{equation*}
Finally, by Markov inequality,
\begin{equation*}
\Pp(R^->t^\alpha(\log t)^4)\leq \Pp((A_t)^c)+\frac{1}{t^\alpha (\log t)^4}\Ep[R^-\ind_{A_t}]
\end{equation*}
is negligible with respect to $(\log t)^{-\kappa}$ for any $\alpha\geq 1-\kappa$ when $0<\kappa<1$, and for any $\alpha>0$ when $\kappa=1$.

\subsection{Proof of Lemma \ref{lemmafailurebound} }

The proof of \eqref{eq:boundM2} is a direct consequence of the definitions of $M_2$ and $\Omega_t.$ 
Then, we shall first prove \eqref{eq:boundVarF}. Since $Var_\omega(F) \le E_{\omega} \left[ F^2 \right],$ 
we shall bound $ E_{\omega} \left[ F^2 \right].$
Recalling \eqref{expF^2} implies
\begin{equation*}
\label{R+majo1}
 R^+ \le C (\log t)^3
\ee^{-\widehat{V}^{\downarrow}(0,e_1)} 
 \max_{0 \le j \le e_1} \ee^{-\widehat{V}(j)},
\end{equation*}
on $\Omega_t.$ To bound $\widehat{V}^{\downarrow}(0,e_1)$ by below, observe first that \eqref{Vcheckeq2} yields
$\widehat{V}^{\downarrow}(0,T_H)\ge V^{\downarrow}(0,T_H)\ge - \alpha
\log t$ on $\Omega_t.$ Moreover, \eqref{def:h(x)} together with
\eqref{Vcheckeq0} imply that $\widehat{V}(y)-\widehat{V}(x)$ is
greater on $\Omega_t$ than
\begin{equation*}
 [V(y)-\max_{y \le j \le e_1-1} V(j)]-[V(x)-\max_{x \le j \le
e_1-1} V(j)]- \log \log t - O(1),
\end{equation*}
 for any $T_H \le x \le y \le e_1,$ which yields $\widehat V^{\downarrow}(T_H,e_1)
\ge -\alpha \log t - \log \log t - O(1)$ on $\Omega_t.$  
Furthermore, since \eqref{def:h(x)} and \eqref{Vcheckeq0} imply
that $\widehat{V}(T_H)$ is larger than $\max_{0 \le j \le
T_H}\widehat{V}(j)-\log \log t-O(1),$ assembling
$\widehat{V}^{\downarrow}(0,T_H)\ge - \alpha \log t$ with
$\widehat{V}^{\downarrow}(T_H,e_1)\ge- \alpha \log t- \log \log t - O(1)$ yields
\begin{equation}
\label{fluctu(b,d):mino1} \widehat{V}^{\downarrow}(0,e_1)\ge-
\alpha \log t- \log \log t - O(1).
\end{equation}
Then, coming back to \eqref{R+majo1}, we have to bound $ \max_{0 \le j \le e_1} \exp\{-\widehat{V}(j)\}.$ Recalling \eqref{Vcheckeq2}, we have $ \min_{0 \le j \le T_H} \widehat{V}(j) \ge  \min_{0 \le j \le T_H} {V}(j) \ge 0,$ by definition of the deep valleys.
Moreover,  it follows from \eqref{fluctu(b,d):mino1} that, for any $T_H \le j \le e_1,$
\begin{align*}
\nonumber
\min_{T_H \le j \le e_1} \widehat{V}(j)
	&= \min_{T_H \le j \le e_1} (\widehat{V}(j)-\widehat{V}(T_H))+\widehat{V}(T_H)\\
	&\ge \widehat{V}^{\downarrow}(T_H,e_1) + h_t \ge h_t - \alpha \log t - \log \log t - O(1),
\end{align*}
which is greater than $0$ for $t$ large enough. Therefore, recalling \eqref{fluctu(b,d):mino1} and \eqref{R+majo1}, we get $ R^+ \le C (\log t)^4 t^{\alpha} $ on $\Omega_t$. 
This result together with the fact that $R^- \le C (\log t)^4 t^{\alpha}$ on $\Omega_t$ concludes the proof of \eqref{eq:boundVarF}.

In a second step, we prove \eqref{eq:boundM1}. To this aim, observe first that $- S_1 \le \widehat M_1- M_1 \le S_2$, where $S_1 \defeq  \sum_{i=0}^{T_{H}-1} |\ee^{-V(i)} - \ee^{-\widehat{V}(i)}|$ and $ S_2   \defeq  \sum_{i=T_{H}}^{e_1-1}  \ee^{-\widehat{V}(i)}$. 
By definition of $\Omega_t$ and since $T_{H} \ge T_{h_t}$ we get $S_2 \le C (\log t) \, \ee^{-{h_t}-\widehat{V}^{\downarrow}(0,e_1)}$ which yields $S_2=o(t^{-\delta}),$ when $t \to \infty$ by recalling \eqref{fluctu(b,d):mino1}. To bound $S_1,$ the definition of $h(\cdot)$ (from the $h$-process) given in Subsection \ref {subsec:hproc}
 implies 
\begin{equation*}
S_1 \le \sum_{i=0}^{T_{H}-1}
  \ee^{-V(i)} (1-h(i))
\le  C (\log t) \ee^{-h_t} \sum_{i=0}^{T_{H}-1} \ee^{\max_{0\le j \le i }V(j)-V(i)},
  \end{equation*}
on $\Omega_t.$ Since $T_H \le e_1 \le C \log t$ on $\Omega_t,$ we obtain $S_1  \le C (\log t)^2 \ee^{-h_t-V^{\downarrow}(0,T_H)} .$
This concludes the proof of \eqref{eq:boundM1} by recalling that $V^{\downarrow}(0,T_H)$ is larger than $-\alpha \log t.$

\subsection{Proof of Lemma \ref{lemmafsuccessbound}  }

Recalling \eqref{expS2}, we get 
$
E_{\omega}[G] \le C (\log  t)^2 \ee^{ \bar V^{\uparrow}(0,e_1)}
$
on $\Omega_t.$
Therefore the proof of Lemma \ref{lemmafsuccessbound}  boils down to finding an upper bound for the largest rise $\bar V^{\uparrow}(0,e_1)$ of $\bar{V}$ inside the interval $[0,e_1]$.
Observe first that \eqref{Vbareq2} allows to bound the largest rise $\bar V^{\uparrow}(T_H,e_1)$ of $\bar{V}$ on the interval $[T_H, e_1]$ by the largest rise of $V$ on this interval, which is less than $\alpha \log t$ on $\Omega_t.$
Concerning the largest rise  of $\bar{V}$ on the interval $[0,T_H]$, we notice, taking into account the small size of the fluctuations of $V$ controlled by $\Omega_t$, that \eqref{def:g(x)} and \eqref{Vbareq0} imply that the difference $\bar{V}(y)-\bar{V}(x)$ is less or equal than
\begin{equation*}
[V(y)-\max_{0 \le j \le y} V(j)]-[V(x)-\max_{0 \le j \le
x} V(j)]+ \log \log t+  O(1),
\end{equation*}
which yields $\bar V^{\uparrow}(0,T_H)
\le \alpha \log t + \log \log t + O(1)$ on $\Omega_t.$ Furthermore, \eqref{Vbareq2} and the fact that $\max_{T_H \le y \le e_1} V(y)\le V(T_H)$ yields $ \max_{T_H \le y \le e_1} \bar V(y)\le \bar V(T_H)$ imply
\begin{equation*}
\bar V^{\uparrow}(0,e_1) \le \max \left\{ \bar V^{\uparrow}(0,T_H) ; \bar V^{\uparrow}(T_H,e_1)   \right\},
\end{equation*}
from which we conclude the proof of Lemma \ref{lemmafsuccessbound}.

\bigskip

\end{document}